\documentclass[11pt,a4paper]{article}%

\RequirePackage[OT1]{fontenc}
\RequirePackage{amsthm,amsmath}
\RequirePackage{amssymb}
\RequirePackage[colorlinks,citecolor=blue,urlcolor=blue]{hyperref}
\usepackage{calc}
\usepackage[round]{natbib}
\usepackage{multirow}
\usepackage[active]{srcltx}
\usepackage{MnSymbol}
\usepackage{graphicx}
\usepackage{verbatim,enumerate}
\usepackage{rotating, lscape}

\def\S{\mathbb{S}}
\def\R{\mathbb{R}}
\def\s{\boldsymbol{s}}

\def\bZ{\boldsymbol{Z}}

\def\bS{\mathcal{S}}
\def\bZ{\boldsymbol{Z}}
\def\bsy{\boldsymbol}
\def\iid{\overset{\mathrm{iid}}{\sim}}

\numberwithin{equation}{section}
\theoremstyle{plain}
\newtheorem{thm}{Theorem}[section]
\theoremstyle{remark}

\begin{document}
\graphicspath{{./figures/}}
\DeclareGraphicsExtensions{.pdf,.png}


\baselineskip=28pt \vskip 5mm
\begin{center} {\LARGE{\bf Modeling Temporally Evolving and Spatially Globally Dependent Data}}
\end{center}

\begin{center}\large
Emilio Porcu,\footnote{\baselineskip=12pt
School of Mathematics and Statistics, University of Newcastle, GB. \\
$\&$  Department of Mathematics, University Federico Santa Maria, 2360102 Valparaiso, Chile.\\
E-mail: georgepolya01@gmail.com. \\
Research work of Emilio Porcu   was partially supported by grant
FONDECYT 1130647 from the Chilean government.
}
Alfredo Alegria \footnote{ \baselineskip=10pt
Department of Mathematics, University Federico Santa Maria, 2360102 Valparaiso, Chile.\\
E-mail: alfredo.alegria.jimenez@gmail.com.
} and 
Reinhard Furrer \footnote{ \baselineskip=10pt
Department of Mathematics and Department of Computational Science, University of Zurich, 8307 Zurich, Switzerland. \\
E-mail: reinhard.furrer@math.uzh.ch.
}

\end{center}

\baselineskip=19pt \vskip 2mm \centerline{\today} \vskip 2mm

\vspace{2cm}

\begin{abstract}\baselineskip=12pt \vskip 2mm \centerline{\today} \vskip 2mm
The last decades have seen an unprecedented increase in the availability of data sets that are inherently global and temporally evolving, from remotely sensed networks to climate model ensembles.
This paper provides a view of statistical modeling techniques for space-time processes, where space is the sphere representing our planet. In particular, we make a distintion between
(a) second order-based, and (b) practical approaches to model temporally evolving global processes. The former are based on the specification of a class of space-time covariance functions, with space being the two-dimensional sphere. The latter are based on explicit description of the dynamics of the space-time process, i.e., by specifying its evolution as a function of its past history with added spatially dependent noise. \\
We especially focus on approach (a), where the literature has been sparse. We provide new models of space-time covariance functions for random fields defined on spheres cross time. 
Practical approaches, (b), are also discussed, with special emphasis on models built directly on the sphere, without projecting the spherical coordinate on the plane. \\
We present a case study focused on the analysis of air pollution from the 2015 wildfires in Equatorial Asia, an event which was classified as the year's worst environmental disaster. 
The paper finishes with a list of the main theoretical and applied research problems in the area, where we expect the statistical community to engage over the next decade.

\end{abstract}



\newpage 

\section{Introduction}
The strong evidence of a changing climate over the last century \citep{ip13} has prompted the scientific community to seek for comprehensive monitoring strategies of the state of the climate system over the entire planet. The surge of satellite observations, the increase of computational and storage availability as well as an increase in the horizontal resolution of global climate models, the deployment of new global and automated network \citep*[the ARGO floats and the AERONET, see][]{aero} and the recent development of smartphone-based data which potentially allow every user on the planet to provide scientific data (Citizen Science) has generated an increase in the size of the data of orders of magnitude. Such an increase in the volume, variate and velocity of globally indexed data serves as a strong catalyst for the statistical community to develop models that are inherently global and time evolving.

The research questions underpinning global space-time models span from optimal interpolation (kriging) for global coverage over both space and time (for variables such as temperature and precipitation, but, also more recently, ozone, carbon dioxide and aerosol optical depth), to interpolation in the input space of Earth System Models by generating fast approximations (emulators) that can be used to perform
 fast and affordable sensitivity analysis. While the aforementioned topics are of high scientific interest, the development of appropriate statistical methodology for global and temporal data has been limited, and has advanced in two seemingly very different directions.

The construction of models on the sphere cross the temporal horizon calls for a rigorous development of a theory that would allow for valid processes with the proper distance over the curved surface of the planet. Under the assumption of Gaussianity, the second order structure can be explicitly specified and the properties of the process can be studied directly from its functional form. The theory for this approach has been actively developed over the last decade, but thus far has been limited to the large-scale structure for the covariance function such as isotropy or stationarity across longitudes. We denote this approach the \textsl{second order-based approach}. 


Alternative definitions of the space-time process rely on either on the de-coupling of the spatial and temporal part through the specification of the dynamics in time, or its representation as a solution of a stochastic partial differential equation (SPDE). These techniques are designed primarily for inferential purposes, and they are particularly suitable for modern massive data sets. This increased suitability for inference, however, comes at at the expense of convenient expressions that allow for an understanding the underlying theoretical properties of the process. We refer to this as the \textit{practical approach}.

We start by describing the second order-based approaches, which rely on the first and second order moment of the underlying random field on the sphere cross time. In particular, the focus becomes the space-time covariance, where space is the two-dimensional sphere. For stochastic processes on a sphere, the reader is referred to \cite{jones}, \cite{marinucci-peccati} and the thorough review in \cite{gneiting2}. For space-time stochastic processes on the sphere, we refer the reader to the more recent approaches in \cite{Porcu-Bevilacqua-Genton, berg-porcu} and \cite{jeong-jun}. Generalizations to multivariate space-time processes have been considered in \cite{alegria}. The increasing interest in modeling stochastic processes over spheres or spheres cross time with an explicit covariance function is also reflected in work in areas as diverse as mathematical analysis \citep{schoenberg, gangolli, hannan, Menegatto-strict, Menegatto-strict2, chen,  menegatto-peron, beatson, guella1, Guella3, Barbosa, Guella2}, probability theory \citep{baldi-marinucci, lang-schwab, Hansen, clarke}, spatial point processes \citep{moller}, spatial geostatistics \citep{papanicolau, gneiting1, hitz, huang, Gerb:Moes:Furr:17}, space-time geostatistics \citep{erosthe2, erosthe, bogaert, Porcu-Bevilacqua-Genton, berg-porcu} and mathematical physics \citep{istas, leonenko, Maliarenko}.  

The natural metric to be used on the sphere is the geodesic or great circle distance (details are explained in subsequent sections). 
However, if this metric is used in space-time covariance models defined on Euclidean spaces cross time, these are generally not valid on the sphere cross time. The resulting need for new theory of space-time covariance functions has motivated a rich literature based on positive definite functions based on great circle distance.

Techniques based on covariance functions are certainly accurate both in terms of likelihood inference and kriging predictions; yet they imply a high computational cost when dealing with large datasets over the space-time domain. The main computational hurdle is the calculation of the determinant and of the quadratic form based on the inverse of the covariance matrix. The so called {\em big $N$ problem} in this case requires a compromise between statistical and computational efficiencies and the reader is referred to, e.g., \cite{bevb:12, Bevilacqua2, stein-jrssb} and \cite{FGN}, for statistical approaches based on the covariance functions, that aim to mitigate such a computational burden. A notable approach for the problem of prediction for very large data sets can be found in \cite{cressie_a3}.  

As for practical approaches, when analyzing modern, massive data sets arising from remotely sensed data, climate models or reanalysis data products, a model that fully specifies the covariance function would require a prohibitive amount of information to be stored in the covariance matrix, as well as a prohibitive number of flops and iterations for maximizing the likelihood or exploring the Markov chain when performing Bayesian inference. 

In the context of data in Euclidean spaces, one of the most popular and natural alternatives is to explicitly describe the dynamics of the process by specifying the evolution as a function of its past history with added spatially dependent noise. This approach has received strong support from reference textbooks in spatio-temporal modelling \citep[see][]{crossi} and recent studies have extended this methodology to the context of spatio-temporal data.  \cite{richardson, richardson2} and \cite{tebaldi-Sanso}, amongst others,  recommend 
the use of an explicit description of the dynamics of the process by specifying its evolution as a function of the spatial distribution of the process. Dynamic spatio-temporal models have a long history in Euclidean spaces \citep[][with the references therein]{crossi}, but the literature on the sphere is more sparse, with the exception of \cite{cas13} and related work. This paper will review the recent literature of this approach, with a particular emphasis on scalable methods for large datasets.  Dynamics on large regions of Earth’s surface have been studied in \cite{cressie_a1}, who consider aerosol data from multiangle imaging spectro radiometers along the Americas, the Atlantic Ocean, and the western part of Europe and North Africa. Other relevant applications are proposed in \cite{oleson, cressie_a2, bgfs}.

In recent years, a powerful modeling approach has emerged based on the consideration of a space-time global process as a solution to a SPDE defined over the sphere and in time. 
Earlier work proposed in \cite{jones-zhang} was based on a diffusion-injection equation, which is just one of a multitude of SPDE-based models commonly used to describe physical processes \citep{erosthe}.
Later studies have focused on specifying an SPDE over the sphere or, more generally, on Riemannian geometries, and more recently under the Integrated Nested Laplace Approximation (INLA) environment \citep{rt:02, Lindgren, rue-inla, cameletti}. The key intuition is to use the SPDE approach as a link to approximate a continuous stochastic process with a Gaussian Markov random field, being a discretized version of the first. This has remarkable computational advantages \citep[see][]{Lindgren} and also allows to build flexible models by providing flexible functional expression of the differential operator. Some notable approaches to model global data under this framework can be found in \cite{bolin}. Although there has been substantial effort to introduce the Markov field architecture coupled with the INLA routine for spatial data, to our knowledge, only few studies have focused on the sphere cross time, which will be surveyed below.


The plan of the paper is the following.  Section 2 provides the necessary background material on random fields over spheres or spheres cross time, with their covariance functions. Section 3 details the \textit{second order-based approach}, with construction principles for characterizing space-time covariance functions on the sphere cross time domain.  New covariance functions are also introduced with formal proofs given in the Appendix. Section 4 is devoted to {\em practical approaches}, with emphasis on dynamical methods as well as methods based on SPDE and Gaussian Markov random fields. Section 5 presents a case study focused on the analysis of air pollution from the 2015 wildfires in Equatorial Asia, an event which was classified as the year's worst environmental disaster. A massive global space-time data of air quality from NASA's MERRA2 reanalysis with more than 12 million data points is provided, and both approaches from (a) and (b) are compared and their relative merits are discussed. Section 6 completes the paper with discussion, and with a list of research problems in the area, where we expect the statistical community to engage over the next decade. Technical proofs are deferred to the Appendix, where we also give some necessary background material, as well as a list of other new space-time covariance functions that can be used on spheres cross time.

\section{Background}

\subsection{Spatial Fields on Spheres, Coordinates, and Metrics}
We consider the unit sphere $\S^2$, defined as $\S^2= \{\s \in \R^3, \| \s \|=1 \}$, where $\|\cdot\|$ denotes Euclidean distance. Every point $\s$ on the sphere $\S^2$ has spherical coordinates $\s=(\phi, \vartheta)$, with $\phi \in [0,\pi]$ and $\vartheta \in [0,2\pi)$ being respectively the polar and the azimuthal angles (equivalently, latitude and longitude). The extension to the sphere with arbitrary radius is straightforward. For planet Earth, the radius is approximately $6,371$\,km, albeit the Earth is not exactly a sphere. 

The natural distance on the sphere is the geodesic or great circle distance, defined as the mapping $d_{{\rm GC}} : \S^2 \times \S^2 \to [0,\pi]$ so that $$d_{{\rm GC}}(\s_1,\s_2)= \arccos \big (\langle \s_1,\s_2 \rangle \big )= \arccos \big ( \sin\phi_1 \sin\phi_2 +\cos \phi_1\cos\phi_2\cos \vartheta \big ),$$ with $\s_i=(\phi_i,\vartheta_i)$, $i=1,2$, and $\langle \cdot,\cdot \rangle$ denoting the classical dot product on the sphere, and where $\vartheta= |\vartheta_1-\vartheta_2|$. Thus, the geodesic distance describes an arc between any pair of points located on the spherical shell.  Throughout, we shall equivalently use $d_{{\rm GC}}(\s_1,\s_2)$ or its shortcut $d_{{\rm GC}}$ to denote the geodesic distance, whenever no confusion can arise.

An approximation of the true distance between any two points on the sphere is the chordal distance $d_{{\rm CH}}(\s_1,\s_2)$, given by 
$$ d_{{\rm CH}}(\s_1,\s_2) = 2 \sin \left ( \frac{d_{{\rm GC}}(\s_1,\s_2)}{2 } \right ), \qquad \s_1,\s_2 \in \S^2,  $$
which defines a segment below the arc joining two points on the spherical shell. 

We consider zero mean Gaussian fields $\{ X(\s), \s \in \S^2 \}$ with finite second order moment. Thus, the finite dimensional distributions are completely specified by the covariance function 
$C_{\bS}: \S^2 \times \S^2 \to \R$, defined by 
$$ C_{\bS}(\s_1,\s_2)= \mathbb{C}{\rm ov} \big ( X(\s_1),X(\s_2) \big ), \qquad \s_1,\s_2 \in \S^2. $$
Covariance functions are positive definite: for any $N$ dimensional collection of points $\{\s_i \}_{i=1}^N \subset \S^2 $ and constants $c_1, \ldots, c_N \in \R$, we have
\begin{equation}
\label{pos-def}  \sum_{i=1}^N\sum_{j=1}^N c_i C_{\bS} \left (\s_i,\s_j \right ) c_j \ge 0,  
\end{equation}
see \cite{bingham}. We also define the variogram $\gamma_{\bS}$ of $Z$ on $\S^2$ as half the variance of the increments of $X$ at the given points on the sphere. Namely,
$$ 2 \gamma_{\bS}(\s_1,\s_2) = \mathbb{V}{\rm ar} \big ( X(\s_2) - X(\s_1)\big ), \qquad \s_1,\s_2   \in \S^2. $$ For a discussion about variograms on spheres, the reader is referred to \cite{huang} and \cite{gneiting2}, and the references therein. 

The simplest process in the Euclidean framework is the isotropic process, i.e., a process that does not depend on a particular direction, but just on the distance between points. We say that $C_{\bS}$ is geodesically isotropic if  $C_{\bS}(\s_1,\s_2)=\psi_{\bS}(d_{{\rm GC}}(\s_1,\s_2))$, for some $\psi_{\bS}:[0,\pi] \to \R$. $\psi_{\bS}$ is called the geodesically isotropic part of $C_{\bS}$ \citep{daley-porcu}. Henceforth, we shall refer to both $C_{\bS}$ and $\psi_{\bS}$ as covariance functions, in order to simplify exposition. For a characterization of geodesic isotropy, the reader is referred to \cite{schoenberg} and the essay in \cite{gneiting2}. The definition of a geodesically isotropic variogram is analogous. 

While a geodesically isotropic process on $\mathbb{S}^2$ is the natural counterpart to an isotropic process in Euclidean space, it is not necessarily an appropriate process for globally-referenced data. Although it may be argued that on a sufficiently small scale atmospheric phenomena lack any structured flow, this does not apply in general for synoptic or mesoscale processes such as prevailing winds, which follow regular patterns dictated by atmospheric circulation. 

As a first approximation for large scale atmospheric phenomena, a process may be assumed to be nonstationary for different latitudes, but stationary for the same latitude. Indeed, is it expected that physical quantities such as temperature at surface will display an interannual variability that is lower in the tropics than at mid-latitude. Therefore, we define the covariance $C_{\bS}$ to be \textit{axially symmetric} if  
\begin{equation} \label{def-axial-symmetry}
 C_{\bS}\left ( \s_1\s_2 \right )  = {\cal C}_{\bS}(\phi_1,\phi_2, \vartheta_1-\vartheta_2), \qquad (\phi_i,\vartheta_i) \in [0,\pi] \times [0,2 \pi) , i=1,2. \end{equation} 
Additionally, an axially symmetric Gaussian field $X(\s)$ is called longitudinally reversible if 
\begin{equation} \label{def-long-rev}
 {\cal C}_{\bS} \big (\phi_{1},\phi_{2},\vartheta \big ) = {\cal C}_{\bS} \big (\phi_1,\phi_2,-\vartheta \big ), \qquad \phi_i \in [0,\pi], \vartheta \in [-2\pi,2\pi), i=1,2. 
 \end{equation}
An alternative notion of isotropy can be introduced if we assume that the sphere $\S^2$ is embedded in the three dimensional Euclidean space $\mathbb{R}^3$, and that the Gaussian process $X$ is defined on $\R^3$. The covariance $C_{\bS}$ can then be defined by restricting $X$ on $\S^2$, which gives rise to the name of Euclidean isotropy or radial symmetry, since it depends exclusively on the chordal distance ($d_{{\rm CH}}$) between the points.  Following \cite{Yadrenko} and \cite{Yaglom2}, for any $C_{\bS}$ being isotropic in the Euclidean sense in $\R^{3}$, the function $C_{\bS}(d_{{\rm CH}})$ is a valid covariance function on $\S^{2}$. This principle has been used to create space-time and multivariate covariance functions, and we come back into this with details in Section \ref{cordal}.

\subsection{Space-Time Random Fields and Covariance Functions}
We now describe zero mean Gaussian fields $\{Z(\s,t), \s \in \S^2, t \in \R\}$, evolving temporally over the sphere $\S^2$. Henceforth, we assume that $Z$ has finite second-order moment.
Given the Gaussianity assumption, we focus on the covariance function $C: \left ( \S^2 \times \R \right )^2 \to \R $,  defined as 
\begin{equation}
\label{cov-def} C\big ( (\s_1,t_1),(\s_2,t_2) \big )= \mathbb{C}{\rm ov} \big ( Z(\s_1,t_1),Z(\s_2,t_2)\big ), \qquad (\s_i,t_i) \in \S^2 \times \R, i=1,2. 
\end{equation}
The definition of positive definiteness follows from Equation (\ref{pos-def}). \\
The covariance $C$ is called separable \citep{GGG07}, if there exists two mappings $C_{\bS} : (\S^2 )^2 \to \R$ and $C_{{\cal T}}: \R^2 \times \R$ such that 
\begin{equation}
\label{separability} 
 C\big ( (\s_1,t_1),(\s_2,t_2) \big ) = C_{\bS}(\s_1,\s_2) C_{{\cal T}}(t_1,t_2), \qquad (\s_i,t_i) \in \S^2 \times \R, \quad i=1,2. 
\end{equation}
Separability can be desirable from a computational standpoint: for a given set of colocated observations (that is, when for every instant $t$ the same spatial sites have available observations), the related covariance matrix factorizes into the Kronecker product of two smaller covariance matrices. However, it has been deemed physically unrealistic, as the degree of spatial correlation is the same at points near or far from the origin in time \citep{GGG07}; a constructive criticism is offered in \cite{stein-jasa}. There is a rich literature on non-separable covariance functions defined on Euclidean spaces; see the reviews  in \cite{GGG07, gregori-review} and \cite{bcg},  with the references therein.  \\

\subsection{Temporally Evolving Geodesic Isotropies and Axial Symmetries}

A common assumption for space-time covariance functions over spheres cross time is that of geodesic isotropy coupled with temporal symmetry: there exists  a mapping $\psi:[0,\pi] \times \R \to \R$ such that 
\begin{equation} \label{isotropy}
C\big ( (\s_1,t_1),(\s_2,t_2) \big )  = \psi \big (d_{{\rm GC}}(\s_1,\s_2),t_1-t_2 \big ), \qquad (\s_i,t_i) \in \S^2 \times \R,  \; i=1,2.
\end{equation} 
Following \cite{berg-porcu}, we call ${\cal P}(\S^2,\R)$ the class of continuous functions $\psi$ such that $\psi(0,0) =\sigma^2 < \infty$ and the identity (\ref{isotropy}) holds. The functions $\psi $ are called the geodesically isotropic and temporally symmetric parts of the space-time covariance functions $C$.  Also, we refer equivalently to $C$ or $\psi$ as covariance functions, in order to simplify the exposition. In Appendix A we introduce the more general class ${\cal P}(\S^n,\R)$ and show many relevant facts about it. 

Equation (\ref{isotropy}) can be generalized to the case of temporal nonstationarity and the reader is referred to \cite{estrade} for a mathematical approach to this problem.

As in the purely spatial case, isotropic models are seldom used in practical applications as they are deemed overly simplistic. Yet, they can served as building blocks for more sophisticated models, that can account for local anisotropies and nonstationarities. 
%
%
%

We now couple spatial axial symmety  with temporal stationarity, so that, for the covariance $C$ in (\ref{cov-def}), there exists a continuous mapping ${\cal C}: [0,\pi]^2 \times [-2 \pi,2 \pi) \times \R$ such that
$$ C\big ( (\s_1,t_1),(\s_2,t_2) \big )  = {\cal C}(\phi_1,\phi_2, \vartheta_1-\vartheta_2, t_1-t_2), \quad\, (\phi_i,\vartheta_i,t_i) \in [0,\pi] \times [0,2 \pi) \times \R, \; i=1,2.$$
Additionally, we call a temporally stationary--spatially axially symmetric random field $Z(\s,t)$, longitudinally reversible, if 
\begin{equation} \label{lingitudinal_reversible}
 {\cal C}(\phi_{1},\phi_{2},\vartheta, u) = {\cal C}(\phi_1,\phi_2,-\vartheta, u), \qquad \phi_i \in [0,\pi], \vartheta \in [-2\pi,2\pi), u \in \R,\; i=1,2. 
 \end{equation}
The use of statistical models based on axially symmetric and longitudinal reversible stochastic processes on the sphere is advocated in \cite{Stein} for the analysis of total column ozone. We show throughout the paper that the construction in Equation (\ref{lingitudinal_reversible}) is especially important for implementing dynamical models as described in Section \ref{dynamical}, where a solely spatial version ${\cal C}_{\bS}(\phi_1,\phi_2,\vartheta)$ is used. Note that a longitudinally reversible covariance function ${\cal C}$ might be separable, in which case 
\begin{equation}
{\cal C}(\phi_1,\phi_2, \vartheta,u)= {\cal C}_{\bS}(\phi_1,\phi_2,\vartheta)\, {\cal C}_{{\cal T}}(u), \qquad (\phi_i,\vartheta,u)\in [0,\pi] \times [-2\pi,2\pi) \times \R,\; i=1,2, 
\end{equation}
with ${\cal C}_{\bS}$ and ${\cal C}_{{\cal T}}$ being marginal covariances in their respective spaces.

\section{Second Order Approaches} \label{construction}
In this section we provide a list of techniques that are used in the literature to implement space-time covariance functions on the two-dimensional sphere cross time.

\subsection{Spectral Representations and Related Stochastic Expansions}
Spectral representations in Euclidean spaces have been known since the work of \cite{schoenberg2} and extended to space-time in \cite{cressie-huang} and then in \cite{gneiting1}. The analogue of spectral representations over spheres was then provided by \cite{schoenberg}. The case of the sphere cross time has been unknown until the recent work of \cite{berg-porcu}: finding a spectral representation for geodesically isotropic space-time covariance function is equivalent to providing a characterization of the class ${\cal P}(\S^2 ,\R)$. Namely, \cite{berg-porcu} establish that a continuous mapping $\psi$ is a member of the class ${\cal P}(\S^2 ,\R)$ if and only if 
\begin{equation} \label{representation-d-schoenberg}
\psi(d_{{\rm GC}},u) = \sum_{k=0}^{\infty} C_{k,{\cal T}}(u) P_{k}(\cos d_{{\rm GC}}), \qquad  (d_{{\rm GC}},u) \in [0,\pi] \times \R,
\end{equation}
where $\{ C_{k,{\cal T}}(\cdot) \}_{k=0}^{\infty}$ is a sequence of temporal covariance functions with the additional requirement that $\sum_{k=0}^{\infty} C_{k,{\cal T}}(0)< \infty$ in order to guarantee the variance $\sigma^2=\psi(0,0)$ to be finite. Here, $P_{k}(x)$ denotes the $k$th Legendre polynomial, $x \in [-1,1]$ \citep[see][for more details]{dai-xu}. We refer to (\ref{representation-d-schoenberg}) as spectral representation of the space time covariance $\psi(d_{{\rm GC}},u) $.

\cite{berg-porcu} showed a general representation for the case of the $n$-dimensional sphere $\S^n$ cross time (see Appendix A for details). This fact is not merely a mathematical artifact, but also the key for modeling strategies, as shown in \cite{Porcu-Bevilacqua-Genton}. 

Some comments are in order.  Clearly, $\psi_{\bS}(d_{{\rm GC}})= \psi(d_{{\rm GC}},0)$ is the geodesically isotropic part of a spatial covariance defined over the  sphere, a characterization of which can be found in the notable work by \cite{schoenberg}; see also the recent review by \cite{gneiting1}. Furthermore, representation (\ref{isotropy}) implies that $\psi$ is separable if and only if the elements $C_{k,{\cal T}}$ of the sequence $\{ C_{k,{\cal T}}(\cdot) \}_{k=0}^{\infty}$ in Equation (\ref{representation-d-schoenberg}) are of the form
$$ C_{k,{\cal T}}(u)= b_{k} \; \rho_{{\cal T}}(u), \qquad u \in \R, \quad k \in \{0,1,\ldots \}, $$ 
with $\{ b_{k}\}_{k=0}^{\infty}$ being a uniquely determined probability sequence, and $\rho_{{\cal T}}$ a temporal covariance function. 
We follow \cite{daley-porcu} and call $\{b_{k} \}_{k=0}^{\infty}$ a $2$-Schoenberg sequence to emphasize that the coefficients $b_k$ depends on the dimension of the two-dimensional sphere where $\psi$ is defined.  

It can be proved that the covariance functions from the series expansion in (\ref{representation-d-schoenberg}) have a corresponding Gaussian process with an associated series expansion as well. Indeed, direct inspection together with the addition theorem for spherical harmonics \citep{marinucci-peccati} shows that (\ref{representation-d-schoenberg}) corresponds to a Gaussian process on $\S^2 \times \R$, defined by the stochastic expansion of \cite{jones}:
\begin{equation}
\label{stochastic}  Z(\s,t)= \sum_{k=0}^{\infty} \sum_{\ell= -k}^{k} A_{k,\ell}(t) {\cal Y}_{k,\ell} (\s), \qquad (\s,t) \in \S^2 \times \R, 
\end{equation}
where ${\cal Y}_{k,\ell}$ are the deterministic spherical harmonics on $\S^2$ \citep{dai-xu}, 
and each of the zero mean Gaussian processes $\{A_{k,\ell}(t) \}_{k=0,\ell=-k}^{\infty,~~k}$ satisfies 
\begin{equation} \label{assumption}
\mathbb{E}\big (A_{k,\ell}(t)A_{k',\ell'}(t') \big )= \delta_{k=k'}\delta_{\ell=\ell'} C_{k,{\cal T}}(t-t'), \qquad t,t' \in \R,
\end{equation} with $\delta$ denoting the classical Kronecker delta, and $\{C_{k,{\cal T}}\}_{k=0}^{\infty}$ a sequence of temporal covariances summable at zero. \\  Unfortunately, the representation does not allow for many interesting examples in closed forms. Hence, using the results in \cite{berg-porcu}, \cite{Porcu-Bevilacqua-Genton} proposed a spectral representation being valid on any $n$-dimensional sphere cross $\R$, namely
\begin{equation}
\label{expansion-infinity} 
 \psi(d_{{\rm GC}},u)= \sum_{k=0}^{\infty} C_{{\cal T}}(u)^k \cos (d_{{\rm GC}})^k, \qquad (d_{{\rm GC}},u) \in [0,\pi] \times \R,
\end{equation}
 where $C_{{\cal T}}:\R \to (0,1]$ is a temporal correlation function. Some examples are reported from \cite{Porcu-Bevilacqua-Genton} in Table~\ref{table1}. Details are explained in Appendix A.

\cite{clarke} take advantage of the stochastic representation (\ref{stochastic}) to study the regularity properties of a Gaussian process $Z$ on $\S^2 \times \R$, in terms of dynamical fractal dimensions and H\"{o}lder exponents. This is achieved by taking a double Karhunen--Lo\`{e}ve expansion, e.g. by expanding each term $A_{k,\ell}(t)$ in terms of basis functions (in particular, Hermite polynomials). Such a double Karhunen--Lo\`{e}ve representation is also the key for fast simulation. \cite{clarke} propose to truncate the order of the double expansions and evaluate the error bound in the L$_{2}$ sense. Figure~\ref{spectral_realization} depicts an example with this method, using $17,000$ points on the sphere and two time instants.

\begin{figure}
\centering
\vspace*{-4mm}
\includegraphics[scale=0.45]{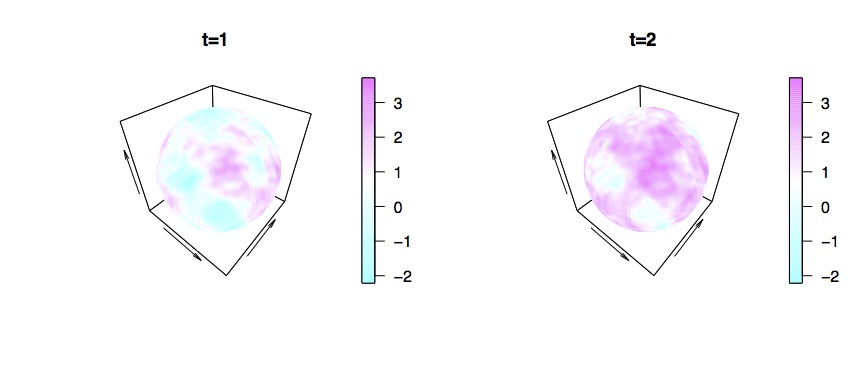}
\vspace*{-14mm}
\caption{Simulated space-time data from the double Karhunen--Lo\`{e}ve expansion used by \cite{clarke}, using 17,000 spatial sites on  $\mathbb{S}^2$ and two temporal instants. The order of trunction is $50$ for both expansions.\label{spectral_realization}}
\end{figure}

Note how in each of the parametric families outlined in Table \ref{table1}, the spatial margin $\psi_{\bS}(d_{{\rm GC}})=\psi(d_{{\rm GC}},0)$ is an analytic function. In particular, we have that the spatial margin is either non differentiable or infinitely differentiable at the origin.  

Spectral representations for stochastic processes over spheres, with the additional feature of axial symmetry, have been proposed in \cite{jones}, \cite{narkovich}, \cite{hitz} and \cite{cas13}. We are not aware of extensions of this representation to space-time, but the work in \cite{jones} suggests that axial symmetry over the sphere coupled with temporal symmetry should be obtained by relaxing the assumption (\ref{assumption}) for the processes $A_{k,\ell}$ in the expansion (\ref{stochastic}). 

An alternative approach to model axially symmetric processes over spheres cross time is proposed in \cite{cas17} on the basis of the stochastic representations in \cite{jones}. 

\begin{table}[!h]  
\caption{Parametric families of covariance functions on $\S^2 \times \R$ obtained through the representation in Equation (\ref{expansion-infinity}). Second column reports the analytic expression, where $g$ is any correlation function on the real line. An additional condition is required for the Sine Power family  \citep[refer to][for details]{Porcu-Bevilacqua-Genton}. All of the members $\psi$ in the second column are rescaled so that $\psi(0,0)=1$. We use the abuse of notation $r$ for the great circle distance $d_{{\rm GC}}$.\label{table1} }
\centering
\medskip
\tabcolsep14pt \linespread{1.4}\selectfont\begin{tabular}{|lll|}
\hline
{ Family} & {Analytic expression} & {Parameters range} \\
\hline
 Negative Binomial & $\psi(r,u)= \Big ( \frac{1-\varepsilon}{  1- \varepsilon g(u) \cos(r)  } \Big ){}^{\tau}$ & $\varepsilon \in (0,1)$, $\tau >0$ \\ 
 Multiquadric & $\psi(r,u)=\Big ( \frac{(1-\varepsilon)^{2}}{  1+ \varepsilon^2- 2\varepsilon g(u) \cos(r) }\Big){}^{\tau}$ & $\varepsilon \in (0,1)$, $\tau >0$ \\ 
 Sine Series & $\psi(r,u)=  \frac{1}{2}{\rm e}^{ g(u)\cos(r)-1} \big ( 1 +  g(u) \cos(r) \big )$  & \\
Sine Power & $\psi(r,u)=1-2^{-\alpha} \Big (1-g(u) \cos(r) \Big ){}^{\alpha/2}$ & $\alpha \in (0,2]$ \\
 Adapted Multiquadric & $\psi(r,u)= \Big ( \frac{(1+g^2(u))(1-\varepsilon )}{ 1+g^2(u) -2 \varepsilon g(u) \cos(r) } \Big){}^{\tau}$ & $\varepsilon \in (0,1), \tau >0$ \\
Poisson & $\psi(r,u)= \exp \big ( \lambda ( \cos(r) g(u) -1 ) \big )$ & $\lambda >0$ \\
\hline
\end{tabular} 
\end{table}

\subsection{Scale Mixture Representations}

Scale mixtures provide a powerful and elegant way to build members of the class ${\cal P}(\S^2,\R)$. We sketch the general principle here and then report some examples. Let $F$ be a positive measure on the positive real line. Let $\psi_{\bS}(d_{{\rm GC}};\xi)$ be a geodesically isotropic spatial covariance for any $\xi \in \R_+$. Also, let $C_{{\cal T}}(u;\xi)$ be a covariance for any $\xi$.  
Then, arguments in \cite{pz:11} and \cite{GGG07} show that the function 
$$ \psi(d_{{\rm GC}},u) = \int_{\R_+} \psi_{\bS}(d_{{\rm GC}};\xi) C_{{\cal T}}(u;\xi) F({\rm d} \xi), \qquad (d_{{\rm GC}},u) \in [0,\pi] \times \R, $$ is a nonseparable geodesically isotropic and temporally symmetric space-time covariance function. 

The scale mixture approach offers a nice interpretation in terms of the construction of the associated process. Let $\Xi$ be a random variable with probability distribution $F$ having a finite first moment. Let $X$ and $Y$ be, respectively, a purely spatial and  a purely temporal Gaussian process. Also, suppose that the random variable ${\Xi}$ and the two processes are mutually independent. Let
$$Z(\s,t \mid \Xi= \xi)= X (\s \xi)Y(t \xi), \qquad (\s,t) \in \S^2 \times \R, $$
where $\xi$ is a realization from ${\Xi}$ and let $Z(\s,t) = \mathbb{E}Z(\s,t \mid \Xi)$, with expectation $\mathbb{E}$ taken with respect to $F$. Then, the covariance of $Z$ admits a scale mixture representation. 
\medskip

In general, it is not possible adapt the classes of nonseparable covariance functions based on scale mixtures and defined over $\R^3 \times \R$ to the case $\S^2 \times \R$. For instance, Gaussian scale mixtures as proposed in \cite{schlather} cannot be implemented on the sphere because the function $C_{\bS}(d_{{\rm GC}};\xi)= \exp(-(d_{{\rm GC}}/\xi)^{\alpha})$ is positive definite on the sphere only when $\alpha \in (0,1]$ \citep[see][]{gneiting2}. Some caution needs be taken when considering the scale mixture based covariances defined on Euclidean spaces by  \cite{porcu-mateu-bevilacqua, porcu-mateu} and \cite{fonseca}, amongst others.

To illustrate some valid examples, further notation is needed. A function $f:[0,\infty) \to (0,\infty)$ is called completely monotonic if it is infinitely differentiable on the positive real line and if its $j$th order derivatives $f^{(j)}$ satisfy $(-1)^j f^{(j)}(t) \ge 0$, $t>0$, $j=0,1,2,\ldots$. 
A scale mixture argument in \cite{Porcu-Bevilacqua-Genton}  shows that the function 
\begin{equation} \label{PBG} \psi(d_{{\rm GC}},u) = \frac{\sigma^2}{\gamma_{{\cal T}}(u)^3} f \big ( d_{{\rm GC}} \gamma_{{\cal T}}(u) \big ), \qquad (d_{{\rm GC}},u) \in [0,\pi] \times \R, \end{equation} is a nonseparable covariance function on $\S^2 \times \R$ provided $f$ is completely monotonic, with $f(0)=1$, and $\gamma_{{\cal T}}$ a variogram, with $\gamma_{\cal T}(0)=1$. Here, $\sigma^2$ denotes the variance. We consider a special case of Equation (\ref{PBG}), namely
\begin{equation}
\label{ex1}
\psi(d_{{\rm GC}},u) = \frac{\sigma^2 }{\left( 1+\frac{|u|}{b_T} \right)^3} \exp\left ( - \frac{ d_{{\rm GC}} \left( 1+\frac{|u|}{b_T} \right) }{ b_S} \right ) ,
\end{equation}
which corresponds to $f(t)=\exp(-t)$, $t \ge 0$, with $\gamma_{\cal T}(u)=(1+|u|)$, $u \in \R$. Here, $\sigma^2$, $b_S$  and $b_T$ are positive parameters associated with the variance, spatial scale and temporal scale of the field, respectively.  Figure \ref{gneiting_realization} shows  a realization using $\sigma^2=1$, $b_S=0.3$ and $b_T = 0.5$.

\begin{figure} 
\centering
\vspace*{-4mm}
\includegraphics[scale=0.45]{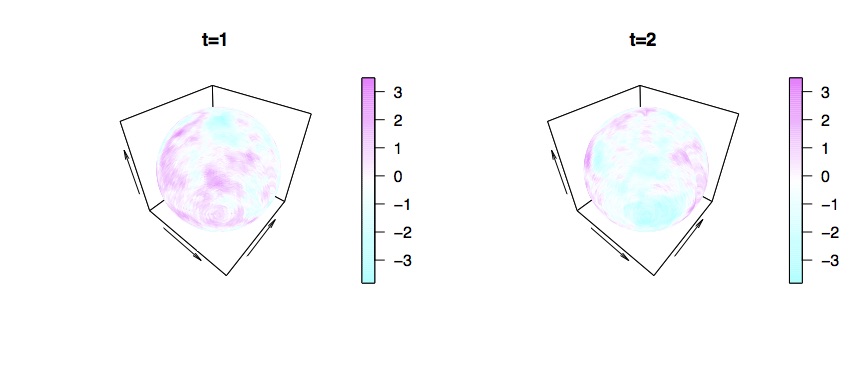}
\vspace*{-14mm}
\caption{Simulated space-time data from the covariance (\ref{ex1}), over 17,000 spatial sites on  $\mathbb{S}^2$ and 2 temporal instants. \label{gneiting_realization}}
\end{figure}

Another scale mixture approach allows for an adaptation of the Gneiting class \citep{gneiting1, zp:11}. Again, \cite{Porcu-Bevilacqua-Genton} show how to derive a space-time covariance of the type 
\begin{equation}
\label{PBG2}  \psi(d_{{\rm GC}},u)= \frac{\sigma^2}{g_{[0,\pi]}(d_{{\rm GC}})^{1/2}} f \left ( \frac{u}{g_{[0,\pi]}(d_{{\rm GC}})}\right ), \qquad (d_{{\rm GC}},u) \in [0,\pi] \times \R, 
\end{equation}
where $f$ is completely monotonic, with $f(0)=1$, and $g_{[0,\pi]}$ is the restriction to the interval $[0,\pi]$ of a function, $g:[0,\infty) \to \R_+$, having a completely monotonic derivative \citep[see][for a description with classes of functions having this property]{porcu-schilling}.  

We are now going to outline a new result within the scale-mixture based approach. Specifically, we consider quasi arithmetic means, as defined in \cite{porcu-mateu-christakos}: these allow to obtain space-time covariances with given margins in space and time. The construction is defined as
\begin{equation}
\label{qarf} 
\psi(d_{{\rm GC}},u) = {\cal Q}_{f} \left ( \psi_{\bS}(d_{{\rm GC}}), C_{{\cal T}}(u) \right ), \qquad (d_{{\rm GC}},u) \in \S^2 \times \R, 
\end{equation}
where ${\cal Q}_{f} (x,y) = f \left ( 1/2 f^{-1}(x) + 1/2 f^{-1}(y) \right )$, for a function $f$ that is completely monotonic on the positive real line, and having a proper inverse, $f^{-1}$, which is always well defined because completely monotonic functions are strictly decreasing. There is a wide list of such functions available and the reader is referred to \cite{SSV} and to \cite{porcu-schilling}. Quasi arithmetic means include as special case all the other means (e.g., the geometric, harmonic, and the Gini means), and the reader is referred to \cite{porcu-mateu-christakos} for a complete description of this framework. 

Using similar arguments as in \cite{porcu-mateu-christakos}, Theorem \ref{THM:QARF} in Appendix C shows the conditions for which $\psi$ in Equation (\ref{qarf}) is a geodesically isotropic space-time covariance function.

\subsection{Space-Time Compact Supports} \label{compact}

Let $\psi$ be a member of the class ${\cal P}(\S^2,\R)$. We say that $\psi$ is dynamically compactly supported on the sphere if there exists a function $h:\R \to (0,c]$, with $c \le \pi$, such that for every fixed temporal lag $u_0$, the function $\psi_{\bS}(d_{{\rm GC}};u_0)= \psi(d_{{\rm GC}},u_0)$ is a covariance function on $\S^2$, and is compactly supported on the interval $[0,h(u_0))$. \\
Let $(x)_+$ denote the positive part of a real number $x$. For $\mu>0$ and $k=0,1,2,\ldots$, we define Wendland functions $\varphi_{\mu,k}$ \citep{Wendland} by 
\begin{equation} \label{wendland}
\varphi_{\mu,k} (x) = \left ( 1 - x \right )_{+}^{\mu+k} {\mathbb P}_k \left ( x \right ), \qquad x \ge 0,
\end{equation}
where ${\mathbb P}_k$ is a polynomial of order $k$. Special cases are depicted in the second column of 
Table~\ref{wendland2}.

\begin{table}[!h] 
\caption{\label{wendland2}Wendland correlation  $\varphi_{\mu,k}(\cdot)$ and Mat{\'e}rn correlation  ${\cal M}_{\nu}(\cdot)$  with increasing smoothness parameters $k$ and $\nu$.
$SP(k)$ means that the sample paths of the  associated  Gaussian field
are  $k$ times differentiable.} 
\bigskip
\centering
{\tabcolsep8pt \linespread{1.4}\selectfont\begin{tabular}{|clclc|}
\hline
~$k$~& $\varphi_{\mu,k}(r)$ &  $\nu$&${\cal M}_{\nu}(r)$ &  $SP(k)$
\raisebox{0pt}[18pt][-7pt]{} \\
\hline
$0$ & $(1-r)^{\mu}_{+}$ & $0.5$&   $e^{-r}$ & 0\\
$1$ & $(1-r)^{\mu+1}_{+}(1+r(\mu+1))$ & $1.5$ & $ e^{-r}(1+r)$ &1 \\
$2$ & $(1-r)^{\mu+2}_{+}(1+r(\mu+2)+r^2(\mu^2+4\mu+3)\frac{1}{3})$ &$2.5$ &$ e^{-r}(1+r+\frac{r^{2}}{3})$&2\\
\multirow{2}{*}{$3$} & $(1-r)^{\mu+3}_{+}
\big( 1+r(\mu+3)+r^2(2\mu^2+12\mu+15)\frac{1}{5}\qquad $
&\multirow{2}{*}{$3.5$}&\multirow{2}{*}{$ e^{-r}(1+\frac{r}{2}+r^2 \frac{6}{15}+\frac{r^3}{15})$}&\multirow{2}{*}{$3$}\\
& $\hspace*{2.8cm} +r^3(\mu^3+9\mu^2+23\mu+15)\frac{1}{15}\big)\,$ &&&\\
\hline
\end{tabular} }
\end{table}

We provide new results in the remainder of this subsection.
Using scale mixture arguments, we can show that 
\begin{equation} \label{wen-space-time}
\psi(d_{{\rm GC}},u)= \frac{\sigma^2}{c^{\alpha}} h(|u|)^{\alpha} \varphi_{\mu,0} \left ( \frac{d_{{\rm GC}}}{ h(|u|)} \right ) = \frac{\sigma^2}{c^{\alpha}} h(|u|)^{\alpha} \left (1- \frac{d_{{\rm GC}}}{ h(|u|)} \right )_+^{\mu} ,
\end{equation}
$(d_{{\rm GC}},u) \in [0,\pi] \times \R$, where, for $\alpha \ge 3$ and $\mu \ge 4$, is a covariance function on $\S^2 \times \R$ provided $h$ is positive, decreasing and convex on the positive real line, with $h(0)=c$, $0<c\le \pi$, and $\lim_{t \to \infty} h(t)=0$. The additional technical restriction on $\mu$ and $\alpha$ is explained in Theorem \ref{THM:ASKEY} in Appendix C, where a formal proof is given. An example can better clarifly things. A potential candidate that can be used as dynamical support in (\ref{wen-space-time}) is the function $h(t)= c(1+t)^{-1/\alpha}$, $t \ge 0$, $\alpha \ge 3$ and $0<c\le \pi$. Then, previous construction becomes 
\begin{equation} \label{wen-space-time2}
\psi(d_{{\rm GC}},u)= \frac{\sigma^2}{(1+|u|)}  \left (1-  \frac{d_{{\rm GC}}}{c(1+|u|)^{-1/\alpha}} \right )_+^{\mu} , \qquad (d_{{\rm GC}},u) \in [0,\pi] \times \R, 
\end{equation}
which shows that the spatial margin $\psi(d_{{\rm GC}},0)$ has compact support $c$. If $c=\pi$, then $\psi(d_{{\rm GC}},0)$ becomes globally supported. Note how the compact support becomes smaller as the temporal lag increases, which is a very intuitive property. \\

The following result characterizes a class of dynamically supported Wendland functions on the sphere cross time, being a generalization of (\ref{wen-space-time}) to arbitrary $k \in \mathbb{N}$.  
\begin{thm} \label{gen_w} 
Let $0<c\le \pi$. Let $h:[0,\infty) \to (0,c]$, be positive, decreasing and convex on the positive real line, with $\lim_{t \to \infty} h(t)=0$. Let $k$ be a positive integer and $\alpha \ge 2k+2$. Let $\varphi_{\mu,k}$ be the Wendland function defined in Equation (\ref{wendland}). Then,
\begin{equation} \label{wen-space-time3}
\psi(d_{{\rm GC}},u)= \sigma^2 h(|u|)^{\alpha} \varphi_{\mu,k} \left ( \frac{d_{{\rm GC}}}{ h(|u|)} \right ), \qquad (d_{{\rm GC}},u) \in [0,\pi] \times \R, 
\end{equation}
is a covariance function on the sphere cross time, provided $\mu \ge k+4$. 
\end{thm}

A more general statement involves compactly supported space-time covariance functions on the circle, where no parametric forms are imposed on the involved functions.

\begin{thm} \label{g_wendland2}
Let $\varphi: \R \times \R \to \R$ be a covariance functions that is symmetric in both first and  second argument, such that $\varphi(x,u)= 0$ whenever $|x| \ge \pi$, for all $u\in \R$. Call $\psi$ the restriction of $\varphi$ to the interval $[0,\pi]$ with respect to the first argument. Then, $\psi$ is a geodesically isotropic and temporally symmetric covariance function on the circle $\S^1$ cross time $\R$.
\end{thm}
The above result is unfortunately insufficient to build for building compactly supported models over spheres cross time. Theorem \ref{gen_w} offers a special case with dynamically supported Wendland functions. To obtain a general assertion on the basis of Theorem \ref{g_wendland2}, we give some hints in the research problems at the end of the manuscript. 

\subsection{Covariance Models under the Lagrangian framework}
Environmental, atmospheric, and geophysical processes are often influenced by prevailing winds or ocean currents \citep{GGG07}. In this type of situation, the general idea of a Lagrangian reference frame applies. 
\cite{alegria2-dimple} considered a simple Lagrangian framework that allows for transport effects over spheres. Namely, they consider random orthogonal $(3 \times 3)$ matrices ${\cal R}$ with determinant identically equal to one, such that ${\cal R}^{-1}= {\cal R}^{\top} $, with $\top$ denoting the transpose operator. Standard theory on random orthogonal rotations shows that $ {\cal R} = Q D Q^{-1}, $
with $Q$ denoting a matrix containing the eigenvectors of ${\cal R}$, and $D$ a diagonal matrix containing the associated eigenvalues. Also, we have that each eigenvalue $\lambda_k$, $k=1,2,3$, can be uniquely written as $\lambda_k = \exp(\imath \kappa_k)$, with $\imath$ being the unit imaginary number and $\kappa_k $ real, for $k=1,2,3$. Then, following \cite{Gachamaker}, one can define the $t$th real power of ${\cal R}$, as 
$$ {\cal R}^{t} = Q \big ({\rm diag}  \left ( \exp (\imath \kappa_k t) \right ) \big ) Q^{-1}. $$ 
Let $X$ be a Gaussian process on $\S^2$ with geodesically isotropic covariance $C_{\bS}(\s_1,\s_2)=\psi_{\bS}(d_{{\rm GC}}(\s_1,\s_2))$. Define
\begin{equation} \label{eqrem}  Z(\s,t)= X\left ( {\cal R}^{t} \s   \right ), \qquad \s \in \S^2, t \in \R.
\end{equation}
Then, the space-time covariance function with transport effect can be expressed as 
\begin{equation} \label{lagrangian2}
 C(\s_1,\s_2,u) = \mathbb{E}  \big ( C_{\bS}(d_{{\rm GC}}({\cal R}^{u} \s_1,\s_2)) \big ) , \qquad \s_1,\s_2 \in \S^2, \quad u \in \R,  \end{equation} 
where the expectation is taken with respect to the random rotation ${\cal R}$. The fact that the resulting covariance is still geodesically isotropic in the spatial component is nontrivial and is shown formally, for some specific choice of the random rotation ${\cal R}$,  in \cite{alegria2-dimple}, at least for the case of the sphere $\S^2$. Some comments are in order. The resulting field $Z$ in Equation (\ref{eqrem}) is not Gaussian (it is Gaussian conditionally on ${\cal R}$ only). Also, obtaining closed forms for the associated covariance is generally difficult. \cite{alegria2-dimple} provide some special cases. 

For the specification of the random rotation ${\cal R}$ in the Lagrangian covariance (\ref{lagrangian2}), various choices can be physically motivated and justified. The reader is referrerd to \cite{GGG07}  for a thorough discussion.  For instance, the random rotation might represent a prevailing wind as in \cite{gupta}. It might be a the westerly wind considered by \cite{haslett} or again, it might be updated dynamically according to the current state of the atmosphere.  
%
%
 Of course, the model (\ref{lagrangian2}) represents only a first step into Lagrangian modeling over spheres. The concept of transport effect should be put into a broader context, e.g. by following the discussion at pages 319--320 of \cite{crossi}, which shows that the Lagrangian model offered here is a very special case, obtained when the spatial field is moving at a constant velocity.  Innovative approaches to space-time data under transport effects can be found in \cite{wikle} and \cite{cressie_a4}.


\subsection{Covariance Models based on Chordal Distances} \label{cordal}

On the basis of the results in \cite{Yadrenko} and \cite{Yaglom2}, \cite{jun-stein, jun08} exploit the fact that, for any covariance function $C_{\bS}$ on $\R^{3}$, the function $C_{\bS}(d_{{\rm CH}})$ is a covariance function on $\S^{2}$. The Mat{\'e}rn function \citep[e.g.,][]{stein-book} is defined as 
\begin{equation} \label{matern}
 {\cal M}_{\nu}(x) = \sigma^2 \frac{2^{1-\nu}}{\Gamma(\nu)}x^{\nu} {\cal K}_{\nu} \left ( x \right ), \qquad x \ge 0, \end{equation}
where $\nu>0$ governs the mean square differentiability of the associated process $Z$ on $\S^2$, which is $k$ times mean differentiable if and only if $\nu > k$. Here, ${\cal K}_{\nu}$ is a modified Bessel function. 

The Mat{\'e}rn model coupled with chordal distance, that is ${\cal M}_{\nu}(d_{{\rm CH}})$,  is valid on $\S^2$ for any positive $\nu$. Unfortunately, arguments in \cite{gneiting2} show that ${\cal M}_{\nu}(d_{{\rm GC}})$ is a valid model on $\S^2$ only for $0 < \nu \le 1/2$, making its use impractical. Space-time models based on geodesic distance inherit this limitation. The same argument of the Mat{\'e}rn covariance is used in \cite{guinness} and \cite{jeong-jun} to assert that the chordal distance might be preferable with respect to the great circle distance.  For instance, a Mat{\'e}rn-Gneiting \citep{gneiting1, pz:11} type covariance function based on chordal distance might be easily implemented. For a positive valued temporal variogram $\gamma_{{\cal T}}$ with $\gamma_{{\cal T}}(0)=1$, the function
\begin{equation} \label{gneiting-matern} C(d_{{\rm CH}},u) = \frac{\sigma^2}{\gamma_{{\cal T}}(u)^{3/2}} {\cal M}_{\nu} \left ( \frac{d_{{\rm CH}}}{ \gamma_{{\cal T}}(u)} \right ), \qquad (d_{{\rm CH}},u) \in [0,2] \times \R
\end{equation}
is a covariance function on $\S^2 \times \R$ for any positive $\nu$. The spatial margin $C(d_{{\rm CH}},0)$ is of Mat{\'e}rn type and keeps all the desirable features in terms of differentiability at the origin. 

The use of chordal distance has been extensively criticized in the literature. For instance, because the chordal distance underestimates the true distance between the points on the sphere, \cite{Porcu-Bevilacqua-Genton} argue that this fact has a nonnegligible impact on the estimation of the spatial scale. Moreover, \cite{gneiting2} argues that the chordal distance is counter to spherical geometry for larger values of the great circle distance, and thus may result in physically unrealistic distortions. Further, covariance functions based on chordal distance inherit the limitations of isotropic models in Euclidean spaces in modeling covariances with negative values. For instance, a covariance based on chordal distance on $\S^2$ does not allow for values lower than $-0.21 \sigma^2$, with $\sigma^2$ being the variance as before. Instead, properties of Legendre polynomials \citep[see][]{szego} imply that correlations based on geodesic distance can attain any value between $-1$ and $+1$. Another argument in favor of the great circle distance is that the differentiability of a given covariance function depending on the great circle distance can be modeled by imposing a given rate of decay of the associated $2$-Schoenberg coefficients $\{b_k\}_{k=0}^{\infty}$, as shown in \cite{moller} or even modeling the rate of decay of the associated $2$-Schoenberg functions in (\ref{representation-d-schoenberg}), as shown by \cite{clarke}. 

A good alternative to the model in Equation (\ref{gneiting-matern}) is the space-time Wendland model based on the great circle distance, as defined  through Equation (\ref{wen-space-time}). Table~\ref{wendland2} highlights a striking connection between the two approaches, showing that Wendland functions allow for a parameterization of the differentiability at the origin in the same way that Mat{\'e}rn does. The compact support of the Wendland functions can imply some problems in terms of loss of accuracy of kriging predictors, as shown in \cite{stein-book}. However, recent encouraging results \citep{bevb:17} show that such a loss is negligible under infill asymptotics.


\subsection{Nonstationary and Anisotropic Space-Time Covariance Functions}

The literature regarding the construction of nonstationary space-time covariance functions is sparse. The works of \cite{jun-stein, jun08} are notable exceptions. The methods proposed by the authors are based on the coupling of the chordal distance with certain classes of differential operators. \cite{estrade} have generalized the Berg-Porcu \citep{berg-porcu} representation of geodesically isotropic-temporally symmetric covariance functions on $\S^d \times \R$. In particular, two generalizations are obtained: an extension of the Berg-Porcu class to the case of temporally nonstationary covariances and a new class that allows for local anisotropy. Anisotropy is also considered in the tour de force by \cite{hitz}, on the basis of chordal distances. Anisotropic components can be induced through the use of Wigner matrices, as explained in \cite{marinucci-peccati}. \\
More recently, \cite{alegria3} have considered geodesically isotropic space-time covariance functions that allow the separation of the linear from the cyclical component in the temporal lag. The authors show that such approach offers considerable gains in terms of predictive performance, in particular in the presence of temporal cyclic components or in the presence of nonstationarities that are normally removed when detrending the data. A recent discussion about nonstationary approaches can be found in \cite{fuglstad} who work under the SPDE framework.

\section{Practical Approaches} \label{practical}

\subsection{Dynamical Approaches} \label{dynamical}


When analyzing massive data sets arising from, for instance,  remotely sensed networks or satellite constellations, climate models, or reanalysis data product, a model that fully specifies the covariance function becomes impractical. Indeed, it would require a prohibitive amount of information to be stored in the covariance matrix, as well as a prohibitive number of flops and iterations for maximizing the likelihood or exploring the posterior distribution. Hence, a compromise between inferential feasibility and model flexibility must be achieved. These limitations arise independently on the space where the random field is defined. For the case of the sphere, the drawbacks are magnified by the considerably more sparse literature on the construction of valid global space-time processes, as shown in the previous section, as well as  by the extremely large size of global data. 

In Euclidean spaces, a very popular approach to drastically reduce the complexity is to separate the spatial and temporal components and to describe the dynamics of the process by specifying its evolution as a function of the past. Variability is then achieved by assuming a random spatial innovation. 
For modeling global climate fields this approach has been further simplified by modeling the temporal dynamics through covariates only \citep{Furr:Sain:Nych:Meeh:07,Gein:Furr:Sain:15}. \\
The dynamical approach has received strong support from reference textbooks in space-time modelling \citep[see e.g.][]{crossi} and recent years have also seen some development of this methodology in the context of global space-time data \citep{cas13}. In this regard the EM algorithm has been efficiently implemented by \cite{finasi1} and successfully used at continental level for multivariate spatio temporal data \citep{finasi2}. Thanks to the Mat{\'e}rn covariance implemented on the sphere, this code works also at the global level.

Let us assume, with no loss of generality,  that the process has zero mean and that it is observed at some locations  $\s_1,\ldots,\s_N \in \S^2$, for equally spaced time points $t$. We denote $\bZ_t=(Z(\s_1,t),\ldots, Z(\s_N,t))^{\top}$ the process at time $t$, and we specify its dynamic through the following recursive equation
\begin{equation}\label{dynam_discrete}
\bZ_t=\mathcal{E}_t(\bZ_{t-1}, \bZ_{t-2}, \ldots, \bZ_{t-p})+\bsy{\varepsilon}_t,
\end{equation}
where $\mathcal{E}_t$ incorporates the evolution of past trajectory of the process up to time $t-p$, and $\bsy{\varepsilon}_t = (\varepsilon(\s_1,t),\ldots,\varepsilon(\s_n,t))^{\top}\iid \mathcal{N}(\bsy{0},\bsy{\Sigma})$ is an innovation vector with purely spatial global covariance matrix $\bsy{\Sigma}$.

The considerable benefit of such an approach is that the spatio-temporal structure of the model is specified by the dynamical evolution $\mathcal{E}$ and the spatial innovation $\bsy{\varepsilon}_t$. Hence, the temporal and spatial part of the model are de-coupled, and this allows for a considerably more convenient inference.  Recent work on satellite data has proposed to couple the dynamical approach dimension reduction techniques, and in particular Fixed Rank Kriging \citep[see][]{cressie_a3, cressie_a2} to further reduce the parameter dimensionality, and to achieve a fit for very large data sets \citep[Fixed Rank Filtering, see][]{cressie_a1, cressie_a4}. While these approaches have been very effective for interpolating large data sets, they do not explicitly account for the spherical geometry, but rather project the data on the Euclidean space. Here, we consider dynamical models for space-time global data with an explicit definition of the covariance function in $\S^2$. 

A particularly appealing class of models with (\ref{dynam_discrete}) is the Vector AutoRegressive model VAR($p$), defined as 
\begin{equation}\label{dynam_var2}
\bZ_t=\sum_{j=1}^p \bsy{\Phi}_j \bZ_{t-j}+\bsy{\varepsilon}_t,
\end{equation} where $\bsy{\Phi}_j$ are  $N\times N$ matrices that encode the temporal dependence at lag $t-j$. This model results in a space-time precision matrix that is block banded, and hence would greatly reduce the storage burden arising from massive data sets as shown in Section \ref{data}.

\subsubsection{The Innovation Structure}

While models such as \eqref{dynam_var2} allow for a substantial computational saving, in typical climate model applications the number of locations is larger than $10,000$. Thus, even the likelihood of the innovation $\bsy{\varepsilon}_t \iid \mathcal{N}(\bsy{0},\bsy{\Sigma})$ is impossible to evaluate for memory issues on a laptop. Therefore, to perform feasible inference on the full data set, additional structure on the model or spatial design of the data must be leveraged on. \cite{cas13} consider the Fourier transform cross longitude of an axially symmetric process, that is
\begin{eqnarray}\label{lo_reg}
		&& \bsy{\varepsilon}_t(\phi,\vartheta) =  \sum_{k=0}^{\infty} e^{\imath \vartheta k}\mathfrak{f}(k;\phi)\widetilde{\varepsilon}_t(k;\phi), \nonumber\\
		&& \nonumber \\ 
		&& \mbox{corr}\big ( \widetilde{\bsy{\varepsilon}}_{t}(k;\phi),\widetilde{\bsy{\varepsilon}}_{t}(k';\phi')\big )  = \delta_{k=k'} \; \rho(k;{\phi,\phi'}), \qquad \phi,\phi'  \in [0,\pi], \vartheta \in [0,2 \pi),
\end{eqnarray} 
\noindent with $\widetilde{\varepsilon}_t(k;\phi)$ being the Fourier process for wavenumber $k$ and latitude $\phi$. Here, $\mathfrak{f}(k;\phi)$ is the spectrum at latitude $\phi$ and wavenumber $k$, and, for any pair of latitudes $(\phi,\phi')$, the function  $\rho(k;{\phi,\phi'})$ defines a spectral correlation (it is also called \textit{coherence}). \cite{jun08} showed that there is a considerable computational benefit if the data are on a $N = N_{\phi} \times N_{\vartheta}$ (latitude $\times$ longitude) regular grid over the sphere and if the process is axially symmetric. Indeed, they showed how the resulting covariance matrix is block circulant and, most importantly, block diagonal in the spectral domain, thus requiring only $O(N_{\phi} \times N_{\vartheta}^2)$ entries to store instead of $O( N^2_{\phi} \times N^2_{\vartheta})$, and $O(N_{\phi} \times N_{\vartheta}^3)$ flops instead of $O(N_{\phi}^3 \times N_{\vartheta}^3)$. \cite{cas13} showed that such structure can be used to perform inference to massive data sets from computer model ensembles (in the range of $10^7$ to $10^9$ data points), by first estimating the spectrum $\mathfrak{f}(k;{\phi})$ for each of the $N_{\phi}$ latitudinal bands in parallel, and then conditionally estimating the structure of the spectral correlation $\rho(k;{\phi,\phi'})$. This approach has been extended to analyze three-dimensional temperature profiles in a regular grid for a data set larger than one billion data points, allowing for an extension of axially symmetric models in three dimensions \citep{cas16}. 

The spectral approach is not just a mere strategy to simplify inference, but also a key to generalize axially symmetric processes to exhibit nonstationarities, by imposing additional structure in the process transformed in the spectral domain, while still guaranteeing positive definiteness of the corresponding covariance functions. In particular, it is still possible to retain the computational convenience of the gridded geometry while assuming nonstationary models cross longitudes. \cite{cas14} explored this idea by assuming that $\widetilde{\bsy{\varepsilon}}_t(k;\phi)$ in \eqref{lo_reg} are correlated across frequencies, with a fully nonparametric dependence structure to be estimated using time replicates, and they showed how the axially symmetric assumption is badly violated for temperature data.  

In order to allow for ocean transitions, \cite{cas17} impose the spectrum $\mathfrak{f}(\cdot;\phi)$ in \eqref{lo_reg} to depend on longitude as well, and call it \textit{evolutionary spectrum} \citep{pries65}. Given two spectra, $\mathfrak{f}_{i}(k; \phi)$, $i=1,2$ and a mapping $b_{{\rm land}}:[0,\pi] \times [0,2\pi) \to [0,1]$, an evolutionary spectrum is attained through the convex combination
\begin{equation}\label{quax_eq1}
\mathfrak{f}(k;{\phi, \vartheta})=\mathfrak{f}_1(k;\phi) b_{{\rm land}}(\phi,\vartheta)+\mathfrak{f}_2(k;\phi)\left\{1-b_{{\rm land}}(\phi,\vartheta)\right\},
\end{equation}
so that $b_{{\rm land}}$ plays the role of modulating the relative contribution of the land regime.
\cite{cas17} showed how this approach is able to capture the majority of the nonstationarity occurring over a single latitudinal band. Recently, \cite{jeong17} proposed an extension to this approach to incorporate mountain ranges in the evolutionary spectrum in the context of wind fields. Even if the data are not on a regular grid, \cite{hor15} showed that it is still possible to interpolate satellite data on a gridded structure using interpolated likelihoods to leverage on spectral methods. They propose to first perform kriging on the original observations, interpolate them over a grid, and evaluate the likelihood of these pseudo-observations. 

The spectral approach to global data allows to achieve a fit for data sets of remarkable size, and allows to generalize axially symmetric models to capture longitudinal nonstationarities. Such a great improvement, however, comes at the cost of a loss of interpretability of the notion of distance. Indeed, the wavenumber $k$ differs in physical length for different latitudes. Thus, interpreting the dependence structure across latitudes is problematic, especially near the poles, where the physical distance among points is very small. Additionally, a process specified with a latitudinally varying spectrum is not, in general, mean square continuous at the poles. Conditions for regularity at these two singular points have been discussed in \cite{cas17}.


\subsubsection{The Temporal Structure}

For sufficiently aggregated data, (\ref{dynam_var2}) can be further simplified by assuming $\bsy{\Phi}_j=$diag($\phi_{ii;j}$), so that the inference can be performed in parallel for each location. Simple diagnostics have shown that this structure is adequate for data spanning from multi-decadal to monthly data, while sub-monthly data would likely require a more sophisticated neighboring-dependent structure. A key feature of some geophysical processes is their dependence along thousands of miles. These teleconnections would require a more complex structure for $\bsy{\Phi}_k$ that would link far away locations.

\subsection{The SPDE Approach}


We start with a brief description of the SPDE approach as in \cite{Lindgren} and \cite{bolin}. We skip all the mathematical details and will stick to the main idea. Clearly, we shall detail our exposition by working on the sphere. The main idea in \cite{Lindgren} is to evade from a direct specification of the Mat{\'e}rn function ${\cal M}_{\nu}$ as defined in Equation (\ref{matern}) in order to be able to work on any manifold, which of course includes the sphere $\S^2$. Taking verbatim from \cite{Lindgren}, {\em our main objective is to construct Mat{\'e}rn fields on the sphere, which is important for the analysis of global spatial and spatiotemporal models.} Also, the authors note that they want to avoid to use the Mat{\'e}rn covariance adapted with chordal distance, in order to avoid {\em the interpretational disadvantage of using chordal distances to determine the correlation between points.}

The solution is to consider the SPDE having as solution a Gaussian field with Mat{\'e}rn covariance function: for a merely spatial process $X$ on $\S^2$, the authors study the SPDE defined through
\begin{equation}
\label{SPDE} 
(\kappa^2 -\Delta  )^{\alpha} X(\s) = {\cal W}(\s), \qquad \s \in \S^2,
\end{equation}
where $\kappa>0$, $\Delta$ is the Laplace-Beltrami operator, and ${\cal W}$ is a Gaussian white noise on the sphere.
Here, the positive exponent $\alpha$ depends on the parameter $\nu$ in (\ref{matern}) as well as on the dimension of the sphere. \\
In order to provide a computationally convenient approximation of (\ref{SPDE}), \cite{Lindgren} find a very ingenious computationally efficient Hilbert space approximation. Namely, the weak solution to (\ref{SPDE}) if found in some approximation space spanned by some basis functions. The computational efficiency is then attained by imposing local basis functions, that is basis functions being compactly supported. This all boils into approximating the field $X$ with a Gaussian Markov field, $\boldsymbol{x}$, with precision matrix $\boldsymbol{Q}$. This idea is then generalized in \cite{bolin} through nested SPDE models.

The SPDE approach proposed in \cite{Lindgren} is ingeniously coupled with hierarchical models by \cite{cameletti} to provide a space-time model. Our exposition is adapted to the domain $\S^2 \times \R$. The authors propose a model of the type 
\begin{eqnarray*}
Z(\s,t)&=& \boldsymbol{\Upsilon}_t^{\top}\boldsymbol{\beta} + W_t(\s) + \varepsilon_t(\s),  \\
W_t(\s)&=& a W_{t-1}(\s) + {\cal W}_{t}(\s), \qquad (\s,t) \in \S^2 \times \R.
\end{eqnarray*}
Here, $\boldsymbol{\Upsilon}_t$ is the vector of covariates and $\boldsymbol{\beta}$ is a parameter vector. The process $\varepsilon_t(\s)$ expresses the measurement error at time $t$ and location $\s$ on the sphere. The latent process $W_t(\s)$ is autoregressive over time and ${\cal W}_t(\s)$ is a purely spatial process with Mat{\'e}rn covariance function as in Equation (\ref{matern}). The bridge with the SPDE approach thus comes from the following assumptions: ${\cal W}$ is approximated through a Gaussian Markov structure with a given precision matrix. The same approximation, but with another precision matrix, is then assigned to the process $W_1(\s)$. Using the dynamic approach, \cite{cameletti} shows that the whole latent process $W$ has a Gaussian Markov structure with a sparse precision matrix that can be calculated explicitly. 
A direct space-time formulation of the SPDE approach is also suggested in \cite{Lindgren}. 
%
\section{Data example} \label{data}

The 2015 El Ni{\~n}o Southern Oscillation (ENSO) was registered as the most intense over the last two decades \citep{wol10}. In September-October 2015, strong ENSO conditions coupled with the Indian Ocean Dipole suppressed precipitation and resulted in dry highly flammable landscape in Equatorial Asia. The extent of the haze from fires in the region was the largest recorded since 1997, and the increased particulate matter concentration over the densely populated area resulted in tens of millions of people being exposed to very unhealthy to hazardous air quality \citep[as defined by the Pollutant Standard Index][]{psi} and resulting into one of the worst environmental disasters on record. The assessment of exposure and mortality from this event is critical for the implementation of future mitigation strategies, and the estimation of these numbers with the associated uncertainty has received widespread media attention \citep{kop16} and even official of estimates from scientific studies from local governments. While it is possible to focus on regional data to provide local exposure estimates, such simulations are very hard to perform and have been attempted only by \cite{crippa16}, while all other studies have been focused on more affordable and readily available data on a global scale. 

Here, we focus on this event using the MERRA2 reanalysis \citep{mol10} data of daily (aggregated from 3-hourly) Aerosol Optical Depth (AOD), see Figure \ref{data-figure1} and movie in the online supplement. Values around 0.01 mean clear sky conditions, 0.4 very hazy condition and around 1 extremely toxic. We have also removed the trend, where for every point the monthly mean (location-wise) has been subtracted. We use this data set to compare different statistical models fitted globally but focused in the region of interest. The practical approaches discussed in Section \ref{practical} would allow for a full estimation over the two months (for a total of more than 12 million points), but that would not be possible with second order approaches. Therefore, we limit our analysis to the first $6$ days of October. While a detailed study would require an additional comparison in terms of exposure maps from population estimates, we decided to avoid this for the sake of simplicity and brevity.

\begin{figure}
\vspace*{-6mm}\hspace*{10mm}
\includegraphics[height=5cm,width=\textwidth-2cm]{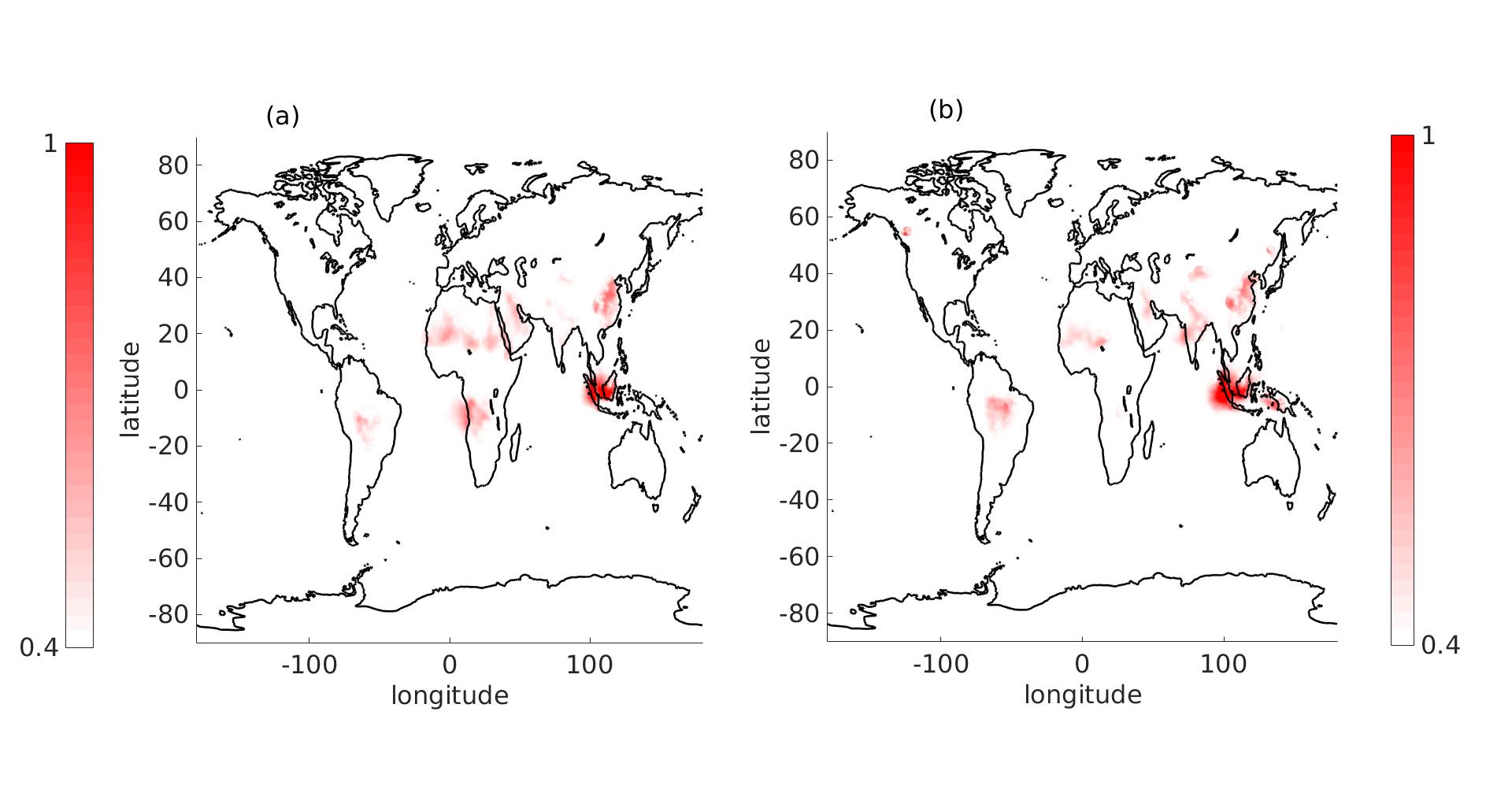}\\[-13mm]
\caption{Averages for AOD for the months of September (a) and October (b). \label{data-figure1}}
\end{figure}


\subsection{The Second Order Approach}

We fit the data set with a second order approach as detailed in Section \ref{construction}. Despite the subsampling in time, the data set is still too large for a full analysis. Hence, a further subsampling in space a grid of 10 degrees in latitude and longitude was performed.  Since several points concentred on the poles, this can produce numerical problems. We thus remove a small portion of observations, which correspond to  latitudes greater than $85$ degrees and lower than $-85$ degrees.  \\
We have considered four models:
\begin{description}
\item[Model 1.] A modified Gneiting covariance, as detailed by Equation (\ref{ex1}).
\item[Model 2. ] A dynamically compactly supported covariance, as in Equation (\ref{wen-space-time2}). 
\end{description}
To explain the choices for Models 3 and 4, we consider the Gneiting \citep{gneiting1} function 
\begin{equation}
\label{ex3}
K(r ,u)  = \frac{\sigma^2 }{\left( 1+\frac{r}{b_S} \right)^3} \exp\left ( - \frac{ |u| }{ b_T \left( 1+\frac{r}{b_S} \right)^{1/2}} \right ).
\end{equation}
\begin{description}
\item[Model 3.] A Gneiting type covariance coupled with the great circle distance, that is $\psi(d_{{\rm GC}},u)=K(d_{{\rm GC}},u))$, with $K$ as in (\ref{ex3}). 
\item[Model 4.] A Gneiting type covariance $K$ as in Equation (\ref{ex3}), coupled with the chordal distance,  $d_{{\rm CH}}$.
\end{description}
The proposed models have three parameters, $(\sigma^2, b_S, b_T)^\top$ indicating the variance, spatial and temporal range respectively. Inference was performed using a pairwise composite likelihood (CL) approach, which provides approximate but asymptotically unbiased estimates to very large data sets. We use the CL method with observations whose spatial distance is less than 6,378 kilometers (equivalent to $1$ radians on a unit sphere).

Table~\ref{estimates} shows CL estimates and the Log-CL value at the maximum, and both models have a similar performance in terms of CL. Figure \ref{variograms} illustrates the empirical spatial (semi) variogram in terms of the great circle distance, at different temporal lags, versus the theoretical models. The models are indeed able to capture this large-scale feature of the data by fitting well the empirical variograms.

\begin{table}[b]
\caption{Composite likelihood estimates, Log-CL value at the maximum and predictive scores. The units for the spatial and temporal scales are km and days, respectively.\label{estimates}}
\renewcommand{\baselinestretch}{1.5}\large\normalsize
\centering
\begin{tabular}{|c|ccccccc|} \cline{2-8} 
\multicolumn{1}{c|}{}& $\widehat{\sigma}^2$  &  $\widehat{b}_S$  & $\widehat{b}_T$   &   Log-CL   & MPSE  &  LSCORE  & CRPS  \\ \hline 
Model 1 & $8.564\times 10^{-3}$ & 816.38 & 6.320&    2223640 &$5.069 \times 10^{-3}$&$-1.052$& $0.107$   \\      
Model 2 & $8.564\times 10^{-3}$ &3845.93 & 6.255&    2223574 &$5.307 \times 10^{-3}$&$-0.935$& $0.106$   \\      
Model 3 & $8.562\times 10^{-3}$ &1651.90 & 2.202&    2223928 &$4.442 \times 10^{-3}$&$-1.288$& $0.117$   \\     
Model 4 & $8.563\times 10^{-3}$ & 765.36 & 5.601&    2223828 &$4.527 \times 10^{-3}$&$-1.281$& $0.118$   \\     
\hline          
\end{tabular}
\end{table}

\begin{figure}[h]
\centering
\includegraphics[scale=0.32]{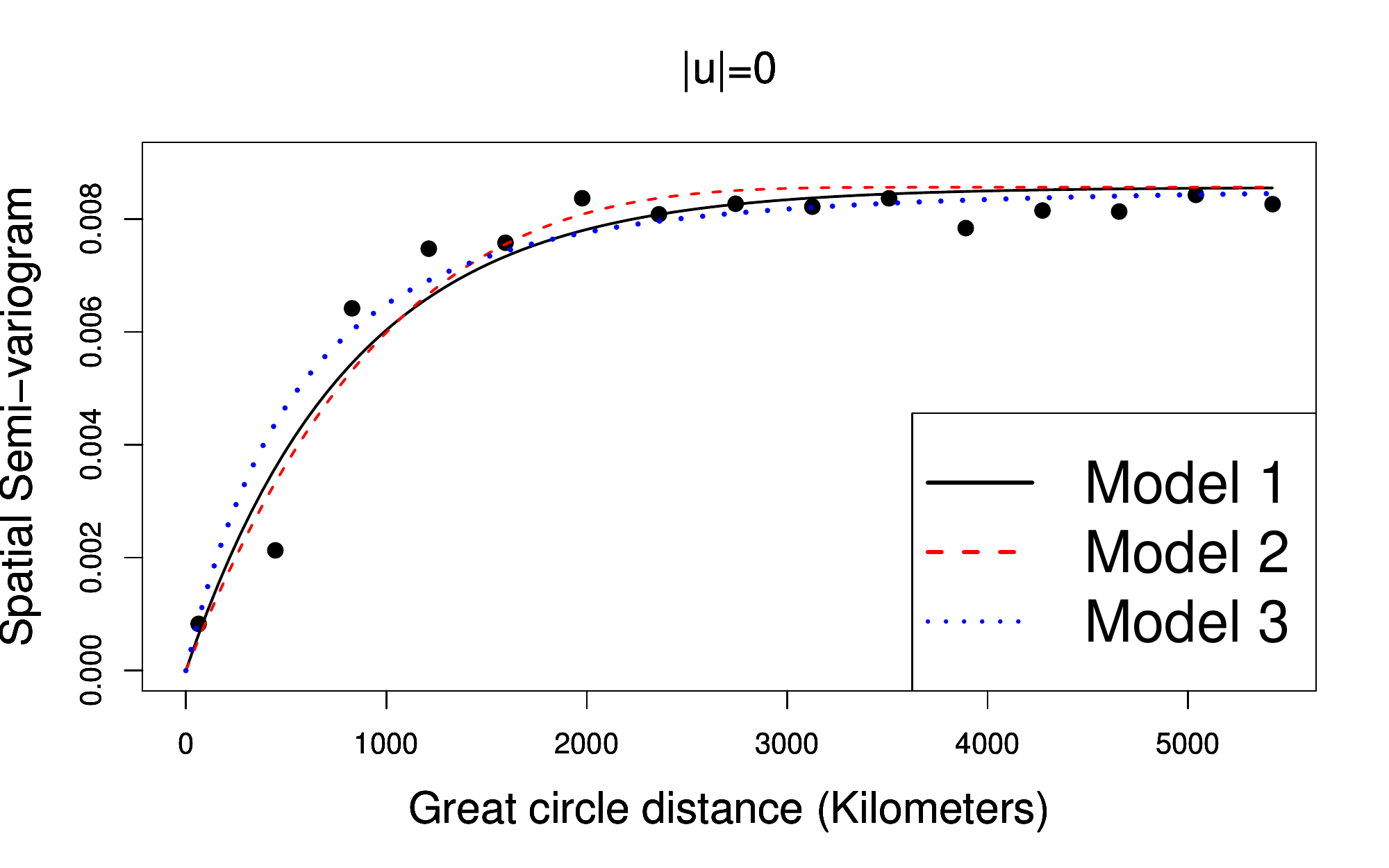} \includegraphics[scale=0.32]{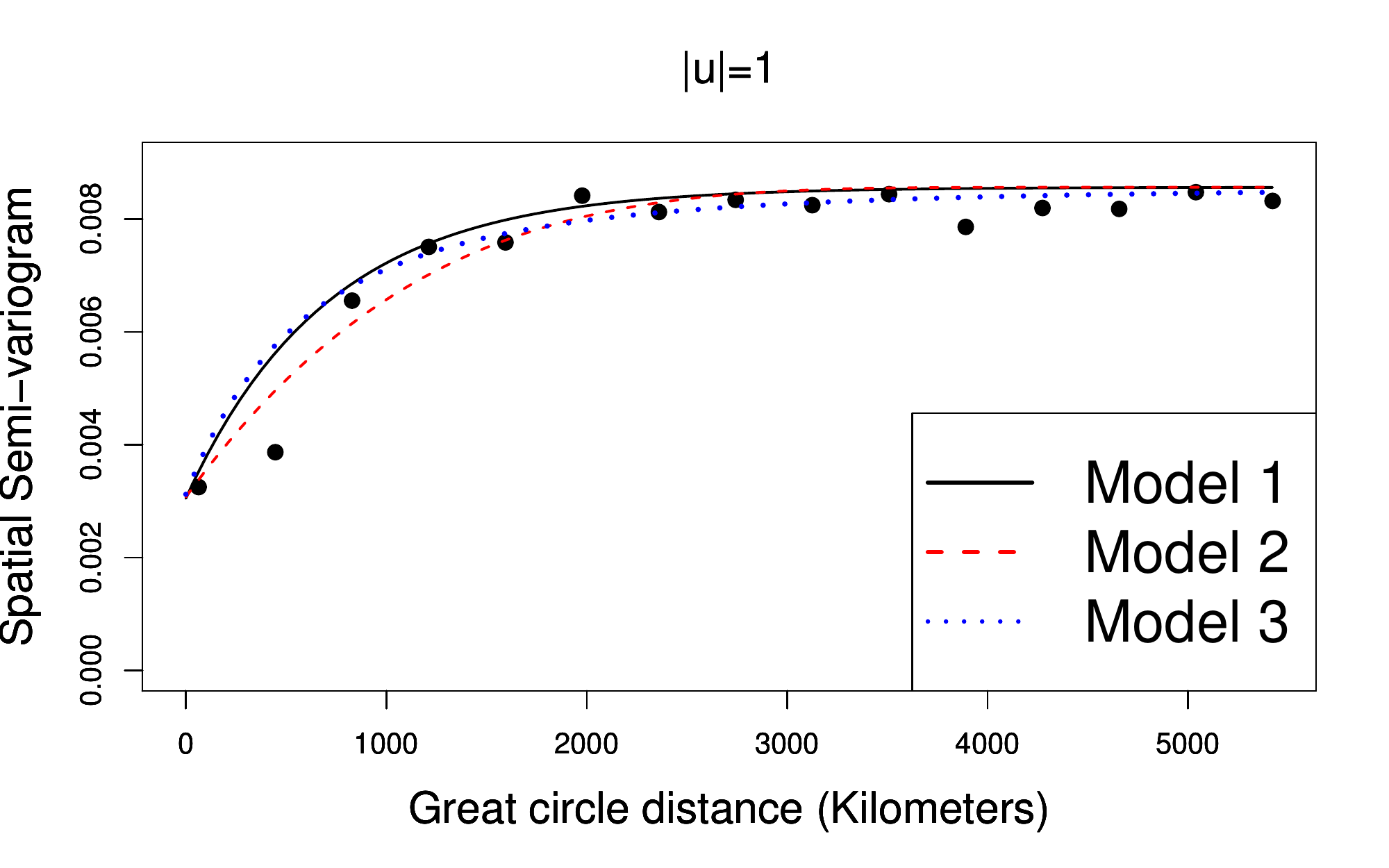}
\caption{Empirical spatial (semi) variograms versus theoretical covariances according to Models 1, 2 and 3,  at different temporal lags. \label{variograms}}
\end{figure}

We now compare the models in terms of their predictive performance. We use the kriging predictor and a drop-one prediction strategy. We consider  the  following indicators: Mean Squared Prediction Error (MSPE), Log-Score (LSCORE) and Continuous Ranked Probability Score (CRPS).  Table~\ref{estimates} contains the indicators for each model. Small values of these indicators suggest better predictions.    

Models 3 and 4 produce better predictive results with respect to Models 1 and 2. In particular,   Model 3 generates an approximate improvement of $2\%$ with respect to Model 4, in terms of MPSE.   Model 3,  based on the great circle distance,   outperforms Model 4 in terms of LSCORE and CRPS.

\subsection{Dynamical approach}

We now fit the same data set with a dynamical model. We choose \eqref{dynam_var2} with $p=1$, i.e., a VAR(1) process, with a diagonal autoregressive structure. Providing a location-specific temporal structure is likely to be too flexible for data subsampled for only 6 days, but further model selection approaches to reduce model complexity were deemed out of the scope of this work. Since the data are gridded, we choose a spectral model \eqref{lo_reg}, with a latitudinally varying spectrum, and a coherence for the same wavenumber, but independence otherwise. While the model was subsampled in time for a comparison with the second-order approach, the inference was scalable and the fit of the entire data set did not require a significant additional computational time. 

Even if the likelihood evaluation can sidestep the big $N$ problem by storing the results in the spectral domain via a Whittle likelihood \citep{whittle}, it is not possible to obtain a maximum likelihood for the entire model. Because the model's structure, comprising of hundreds of parameters, the optimization over the entire space would be impossible. Therefore, the estimation is performed step-wise. First, the temporal dependence is estimated. Then, we estimate the longitudinal dependence, then the latitudinal dependence, and hence the parameter estimates are to be regarded as local approximations.

\begin{figure}
\centering\vspace*{-2mm}
\includegraphics[scale=0.3]{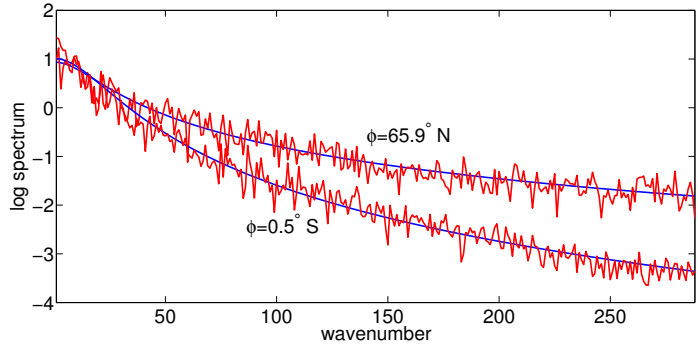}
\caption{Fitted parametrical (blue) and non parametrical (i.e., periodogram red) log spectrum for two different latitudinal bands, one at the equator, one at high northern latitudes. The periodogram was obtained by averaging across all longitudes and times.  \label{fit_spectrum}}
\end{figure}
The spectral model allows for a different spectral shape at different latitudes, and it  is flexible enough to capture very different behaviours, as illustrated in Figure~\ref{fit_spectrum}. Indeed, AOD residuals near the equator display a much smoother behaviour than at high latitudes, where the spectrum is more flat, but the model is able to fit both behaviours naturally.

\subsection{Comparison for Equatorial Asia}

While both models are fitted globally, the interest lies in their relative performance in the area of interest, i.e., in Equatorial Asia.
Table~\ref{mod_fit} shows a comparison of the aforementioned models in terms of the marginal likelihood in this area, defined as all points whose latitudes is between $-11.3$ and $15.3$ degrees ($N_\phi=53$) and whose longitude is between $93$ to $137$ degrees ($N_\vartheta=71$). The dynamically specified model outperforms almost uniformly Models 1 and 2 in the second order approach: it is faster, richer in complexity yet achieving a overwhelmingly better fit, so it is clearly more suitable for a more detailed analysis in the area.

While these results are strongly in favor of the dynamically approach, it must be pointed out that the particular setting of the global space-time data analyzed, i.e., a regular grid in space and equal observations in time, is such that a spectral model is clearly the best choice. An application with more complicated geometries such as with satellite data, would have required an adaptation of spectral methods along the lines of \cite{hor15} and, likely, a worse fit. While such a comparison with an irregular design would have been added to the discussion, we decided to avoid it for the sake of brevity. 


\renewcommand\arraystretch{1.3}
\begin{table}[tbhp]
\caption{{\small Comparison between different models in terms of number of parameters (excluding the temporal ones), computational time (minutes), normalized loglikelihood, and Bayesian Information Criterion for $K=6$ days from October 1st to October 6th 2015.}}\label{mod_fit}
\centering
{\begin{tabular}{|c|cccc|}
\cline{2-5} 
\multicolumn{1}{c|}{}    & $\#$ param   &Time  (minutes)  & $\Delta$-loglik$/(N_{\phi}N_{\vartheta}K)$  & BIC$\times 10^4$ \\\hline
Model 1             & $3$   & $238$  & $-3.24$ & $-4.50$  \\
Model 2             & $3$   & $236$  & $-3.24$ & $-4.55$  \\
Model (\ref{lo_reg})& $161$ & $4.2$     & $0$  & $-19.0$  \\ \hline
\end{tabular}}
\end{table}

\section{Research Problems}

This section is devoted to a list of research problems that we consider relevant for future researches. 
\begin{description}
\item[1. Nonstationary Space-Time Covariances.] The models proposed in Section 3 are geodesically isotropic and temporally symmetric. Since real phenomena on the globe are notably nonstationary, it is necessary to use those models as building blocks for more sophisticated constructions. In particular, the development of nonstationary models is necessary and a promising direction of research is to extend the work of \cite{guella1, Guella2, Guella3} who define kernels over products of $n$-dimensional spheres. 
Another direction of research might be to take into account differences in local geometry of the sphere representing planet Earth. Thus, the natural solution is to work on processes evolving temporally over Riemannian manifolds. In this direction, the works of \cite{menegatto-peron} and  the recent work by \cite{Barbosa} might be very useful. 
\item[2. Models based on Convolutions.] Convolution arguments have been used for Gaussian processes defined over spheres (no time) in order to index the associated fractal dimension. \cite{Hansen} show that Gaussian particles provide a flexible framework for modelling and simulating
three-dimensional star-shaped random sets. The authors show that, if the kernel is a von Mises-Fisher
density, or uniform on a spherical cap, the correlation function of the associated
random field admits a closed form expression. We are not aware of any convolution argument for space-time covariance functions. The work of \cite{ziegel} might be very useful in this direction. 
\item[3. Physical-based Constructions.] Constructions based on dynamical models or moving averages have been proposed by \cite{Baxevani} and \cite{schlather}. We are not aware of such extensions to the case $\S^2 \times \R$, but certainly it would be worth being studied.  Some other constructions based on physical characteristic of the space-time process might be appealing for modeling several real processes. In this direction, it would be desirable to study the approaches proposed in \cite{erosthe} as well as \cite{kolovos1, erosthe2, santili} and \cite{kolovos2}. 
\item[4. Multivariate Space-Time Models.] Often, several variables are observed over the same spatial location and same temporal resolution. Thus, there is substantial need of space-time multivariate covariance models being geodesically isotropic or axially symmetric in the spatial component. The literature in this subject is very sparse, with the notable exceptions of \cite{jun08} and \cite{alegria}. 
\item[5. Mat{\'e}rn Analogues over Spheres.] Find the counterpart of the Mat{\'e}rn covariance function on the sphere. This would allow to have a covariance function that indexes the differentiability at the origin of the associated Gaussian field. Notable attempts have been made by \cite{guinness}, with partial success. \cite{moller} suggest to model the $n$-Schoenberg coefficients (see Appendix A for details) imposing a given rate of decay, but this does not allow for explicit closed forms. 
\item[5. Compact Supports with Differentiable Temporal Margins.] 
A potential drawback of the space-time construction as in Theorem \ref{gen_w} is that it only allows for temporal margins being non differentiable at the origin. An important step ahead would be to improve such a limitation.  Following the arguments in the proof of Theorem \ref{gen_w}, the crux would be in proving that, for some mapping $h:[0,\infty) \to \R$, the function
$$ h(u^2)^{\alpha+d} \mathstrut_1F_{2}\bigg \{1+k; \frac{\mu+ 2}{2}+k, \frac{\nu+ 3}{2}+k; -n^2 h(u^2)^2/4 \bigg \}, \qquad u \in \R, \quad n \in \mathbb{N}, $$
is positive definite for all $n \in \mathbb{N}$. 
\def\bomega{\boldsymbol{\omega}}
\item[6. Spectral Constructions {\em {\`a} la Stein}.]
\cite{stein-jasa} proposes a class of spectral densities in $\R^d \times \R$ of the type
$$ \mathfrak{f}(\bomega,\tau) = \left ( (1+\|\bomega\|)^{\alpha_1} + (1+ |\tau|)^{\alpha_2} \right )^{-\alpha_3}, \qquad (\bomega,\tau) \in \R^d \times \R, \quad \alpha_i>0, $$
for $i=1,2,3$.   Under the condition $\alpha_3 < \alpha_1/d+\alpha_2$, $\mathfrak{f}$ is in $L_1(\R^{d}\times \R)$. The partial Fourier transforms allow for indexing the differentiability at the origin in a similar way as the Mat{\'e}rn covariance function. Further, the resulting covariance is smoother away from the origing than at the origin, which is important as reported in Theorem 1 in \cite{stein-jasa}. 
There is no analogue of $\mathfrak{f}$-based constructions for the class ${\cal P}(\S^2, \R)$. One should necessarily start from Berg-Porcu \citep{berg-porcu} characterization and try to find a related approach that allows to obtain such a construction on spheres cross time. 
\item[7. Local Anisotropies and Transport Effects.]
The geodesically isotropic models should be extended to allow for local anisotropies as in \cite{paciorek}. In this direction,  a major step should be made in order to generalize the Lagrangian model in Equation (\ref{lagrangian2}) to the nonstationary case.
\item[8. Strictly Positive Definite Functions.] 
A function $C: (\S^2 \times \R)^2 \to \R$ is called strictly positive definite when inequality 
$$  \sum_{i,j}^N c_i C\big ( (\s_i,t_i),(\s_j,t_j) \big ) c_j >0,  $$
holds for any $\{ c_k \}_{k=1}^N \subset \R$, unless $c_1=\dots=c_N=0$, and $\{\s_k,t_k\}_{k=1}^N \subset \S^d \times \R$. 
There is substantial work on strict positive definiteness and the reader is referred to \cite{chen, Menegatto-strict, Menegatto-strict2} and \cite{menegatto-peron}. A characterization of the subclass of ${\cal P}(\S^n,\R)$ (see Appendix A) with members $\psi$ such that the corresponding covariances $C$ are strictly positive definite is still elusive.  
\item[9. Walks through Dimensions.]
The literature on walks through dimensions is related to operators that allow, for a given positive definite function on the $n$-dimensional sphere, to obtain new classes of positive definite functions on $n'$-dimensional spheres, with $n \ne n'$. The application of such operators has consequences on the differentiability at the origin of the involved functions.   
Walks on spheres have been proposed by \cite{beatson}, \cite{ziegel} and by \cite{massa}, this last extending previous work to the case of complex spheres. It would be timely to obtain walks through dimensions for the members of the classes ${\cal P}(\S^n,\R)$, for $n$ a positive integer.  
\item[10. Optimal Prediction on Spheres cross Time.] Much work needs to be done in order to assess asymptotic optimal kriging prediction over spheres cross time when the covariance is missspecified. The work of \cite{arafat} opens a bridge between equivalence of Gaussian measures and asymptotic optimal prediction over spheres. A relevant direction of research would be to evaluate the screening effect on spheres cross time. After the works of \cite{stein-book}, there is nothing related to spectral behaviours over spheres that allow for evaluating the corresponding screening effect. 
\item[11. Fast Simulations on Spheres.] \cite{lang-schwab} and \cite{clarke} propose fast simulations through truncation of Karhunen-Lo\`{e}ve expansions. It would be very interesting to propose analogues of circulant embedding methods, where the difficulty is mainly due to the fact that it is not possible to set a regular grid on the sphere. 
\item[12. Chordal vs Great Circle Again.] Moreno Bevilacqua (personal communication) points out that he has tried the following experiment. Simulate any set of points on the sphere cross time. Take any element $\psi$ from the class ${\cal P}(\S^2,\R)$. Replace the great circle distance with the chordal distance and calculate the corresponding matrix realizations. With all the experiments, he always found strictly positive eigenvalues. Thus, the question is: suppose that $\psi \in {\cal P}(\S^2,\R)$. Then, is it true that $\psi (d_{{\rm CH}})$ is positive definite on $\S^2 \times \S^2 \times \R$? 
\end{description}

\section*{Appendix}

\subsection*{Appendix A: Mathematical Background}
We need some notation in order to illustrate the following sections. 
This material follows largely the exposition in \cite{berg-porcu}. 

Let $n$ be a positive integer. We denote $\S^n$ the $n$-dimensional unit sphere of $\R^{n+1}$, given as
\begin{equation}\label{eq:sphere}
\mathbb S^n=\{\s \in\mathbb R^{n+1}\mid \|\s\|=1\}, \;n\ge 1.
\end{equation} 
We also consider 
$$
\S^\infty=\{(s_k)_{k\in\mathbb{N}}\in \R^{\mathbb{N}} | \sum_{k=1}^\infty s_k^2=1\},
$$
which is the unit sphere in the Hilbert sequence space $\ell_2$ of square summable real sequences.
Following \cite{berg-porcu}, we thus define the class ${\cal P}(\S^n,\R)$ as the class of continuous functions $\psi:[0,\pi] \times \R \to \R$ such that $$C\big ((\s_1,t_1),(\s_2,t_2) \big ) = \psi(d_{{\rm GC}}(\s_1,\s_2),t_1-t_2), \qquad (\s_i,t_i) \in \S^n \times \R, \quad i=1,2,  $$
where $d_{{\rm GC}}$ as already been defined as the great circle distance. \cite{berg-porcu} define 
${\cal P}(\S^{\infty},\R)$ as $\bigcap_{n \ge 1} {\cal P}(\S^n,\R)$. The inclusion relation 
$$ {\cal P}(\S^1,\R) \supset {\cal P}(\S^2 ,\R) \supset \cdots \supset {\cal P}(\S^{\infty},\R) $$
is strict and the reader is referred to \cite{berg-porcu} for details. 

We recall the definition of Gegenbauer polynomials $C_k^{(\lambda)}$,
given by the generating function \citep[see][]{dai-xu}
\begin{equation}\label{eq:Geg}
(1-2xz +r^2)^{-\lambda}=\sum_{k=0}^\infty C_k^{(\lambda)}(z) r^k,\quad |r|<1, z\in \mathbb C.
\end{equation}
Here, $\lambda>0$. We have the classical orthogonality relation:
\begin{equation}\label{eq:orth}
\int_{-1}^1 (1-x^2)^{\lambda-1/2}C_k^{(\lambda)}(x)C_h^{(\lambda)}(x)\,{\rm d}x=
\frac{\pi \Gamma(k+2\lambda) 2^{1-2\lambda}}{\Gamma^2(\lambda)(k+\lambda) k!}\delta_{h=k},
\end{equation} with $\Gamma$ denoting the Gamma function.  \\
It is of fundamental importance that $
|C_k^{(\lambda)}(x)|\le C_k^{(\lambda)}(1)$, $x\in[-1,1]$. 
The special value
$\lambda=(n-1)/2$ is relevant for the sphere $\mathbb S^n$ because of the relation to spherical harmonics, which is illustrated in \cite{berg-porcu} as follows. A spherical harmonic of degree $k$ for $\mathbb S^n$  is  the restriction to $\mathbb S^n$ of a real-valued harmonic homogeneous polynomial in $\mathbb R^{n+1}$ of degree $k$. Together with the zero function, the spherical harmonics of degree $k$  form a finite dimensional vector space denoted $\mathcal H_k(n)$. It is a subspace of 
the space  ${\mathcal C}(\mathbb S^n)$ of continuous functions on $\mathbb S^n$.  We have
\begin{equation}\label{eq:dim}
N_k(n):=\dim \mathcal H_k(n)=\frac{(n)_{k-1}}{k!}(2n+n-1),\;k\ge 1,\quad N_0(n)=1,
\end{equation} 
\citep[see][p.3]{dai-xu}.  Here, $(n)_{k-1}$ denotes the Pochammer symbol. \\
The surface measure of the sphere is denoted $\omega_n$, and it has total mass
\begin{equation}\label{eq:mass} 
||\omega_n||=\frac{2\pi^{(n+1)/2}}{\Gamma((n+1)/2)}.
\end{equation}
The spaces $\mathcal H_k(n)$ are mutual orthogonal subspaces of the Hilbert space $L^2(\mathbb S^n,\omega_n)$, which they generate. This means that any $F\in L^2(\mathbb S^n,\omega_n)$ has an orthogonal expansion
\begin{equation}\label{eq:exp}
F=\sum_{k=0}^\infty S_k,\, S_k\in\mathcal H_k(n),\quad ||F||_2^2=\sum_{k=0}^\infty ||S_k||_2^2,
\end{equation}
where the first series converges in $L^2(\mathbb S^n,\omega_n)$, and the second series is  Parseval's equation \citep{berg-porcu}. Here $S_k$ is the orthogonal projection of $F$ onto $\mathcal H_k(n)$  given as
\begin{equation}\label{eq:proj}
S_k(\xi)=\frac{N_k(n)}{||\omega_n||}\int_{\mathbb S^n}c_k(n, \langle \s,  \eta \rangle )F(\eta)\,{\rm d}\omega_n(\eta).
\end{equation}
See the addition theorem for spherical harmonics, \citep{schoenberg}.
For $\lambda= (n-1)/2$, $c_k(n,x)$ is defined
as  the normalized Gegenbauer polynomial 
\begin{equation}\label{eq:nor}
c_k(n,x)=C_k^{((n-1)/2)}(x)/C_k^{((n-1)/2)}(1)=\frac{k!}{(n-1)_k}C_k^{((n-1)/2)}(x),
\end{equation}      
being 1 for $x=1$.\\
Specializing the orthogonality relation (\ref{eq:orth}) to $\lambda=(n-1)/2$ and using Equations (\ref{eq:dim}), (\ref{eq:mass}), \cite{berg-porcu} get for $n\in\mathbb{N}$
\begin{equation}\label{eq:orthspec}
\int_{-1}^1 (1-x^2)^{n/2-1}c_k(n,x)c_h(n,x)\,{\rm d}x=
\frac{||\omega_n||}{||\omega_{n-1}||N_k(n)}\delta_{h = k}.
\end{equation}
The following result characterizes completely the class ${\cal P}(\S^n ,\R)$.
\begin{thm}{\rm \citep{berg-porcu}.} \label{thm:main} Let $n\in\mathbb{N}$  and let $\psi:[0,\pi]\times \R \to \mathbb C$ be a continuous function. Then
$\psi$ belongs to the class ${\cal P}(\S^n,\R)$ if and only if there exists
 a sequence $\{ \varphi_{k,n} \}_{k=0}^{\infty}$ of positive definite functions on $\R$, with $\sum \varphi_{k,n}(0)<\infty$, such that
such that
\begin{equation}\label{eq:expand} \psi
(d_{{\rm GC}},u)=\sum_{k=0}^\infty \varphi_{k,n}(u) c_k(n,\cos d_{{\rm GC}}),
\end{equation}
and the above expansion is uniformly convergent for $(d_{{\rm GC}},u)\in [0,\pi]\times \R$.
We have
\begin{equation}\label{eq:coef}
\varphi_{k,n}(u)=\frac{N_k(n)||\omega_{n-1}||}{||\omega_n||}\int_{0}^{\pi} \psi(x,u)c_k(n, x) \sin x^{n-1}  \,{\rm d}x.
\end{equation}
\end{thm}
We also report the characterization of the class ${\cal P}(\S^{\infty},\R)$ obtained by the same authors. 

\begin{thm}\citep{berg-porcu} \label{thm:main2} Let 
$\psi:[0,\pi]\times \R \to \R$ be a continuous function. Then
$\psi$ belongs to ${\cal P}(\S^\infty,\R)$ if and only if there exists
 a sequence $\{ \varphi_{k} \}_{k=0}^{\infty }$ of positive definite functions on $\R$, with $\sum_k \varphi_{k}(0)<\infty$ 
such that
\begin{equation}\label{eq:expandp}
\psi(d_{{\rm GC}},u)=\sum_{k=0}^\infty \varphi_{k}(u) \cos^k d_{{\rm GC}},
\end{equation}
and the above expansion is uniformly convergent for $(d_{{\rm GC}},u)\in [0,\pi]\times G$.
\end{thm}
Some comments are in order. Theorems \ref{thm:main} and \ref{thm:main2} are fundamental to create the examples illustrated in Table \ref{table1}. Also, they are crucial to solve some of the open problems that we reported in Section 6. 
Table~\ref{table-examples} lists some permissible models on $\S^2 \times \R$.


\begin{table} 
\caption{{\small Parametric families of geodesically isotropic space-time covariance functions. {PMC} is the acronym for Porcu-Mateu-Christakos \cite{porcu-mateu-christakos}. {PBG} is the acronym for Porcu-Bevilacqua-Genton \cite{Porcu-Bevilacqua-Genton}. We use the abuse of notation $r$ for the great circle distance $d_{{\rm GC}}$.} \label{table-examples} }
\centering
{\small\linespread{1.2}\selectfont
\begin{tabular}{|cc|c|}
\hline 
Family & Expression for $\psi(r,u)$ & Parameter Restrictions \\
\hline 
Fonseca-Steel & $\displaystyle\frac{ 1 + (\gamma_{\bS}(r)+\gamma_{{\cal T}}(u))}{ ( 1+\gamma_{\bS}(r) \big )^{\lambda_1}} \frac{{\cal K}_{\lambda_1} \left ( 2 \sqrt{(a+\gamma_{\bS}(r))^{\delta}} \right )}{{\cal K}_{\lambda_1}(2\sqrt{a \delta})} $ &  $a_1,\delta,\lambda_i >0$, $i=0,1$.\raisebox{0pt}[30pt][18pt]{~} \\
\hline
Clayton PMC & $\bigg ( (1+r^{\alpha})^{\rho_1} +(1+|u|^{\beta})^{\rho_2} -1 \bigg )^{-\rho_3}$ &  $\alpha,\beta,\rho_1,\rho_2 \in (0,1]$; $\rho_3 >0$.\raisebox{0pt}[26pt][15pt]{~} \\
\hline 
Gumbel PMC & $ \exp \bigg ( -\left ( r^{\alpha_1} + u^{\alpha_2} \right )^{\alpha_3} \bigg )$ &  $\alpha_i \in (0,1]$, $i=1,2,3$.\raisebox{0pt}[26pt][15pt]{~} \\
\hline    
PBG & \multicolumn{1}{c}{ Equation (\ref{PBG})} &  \\
\multicolumn{3}{|c|}{ \linespread{2.2}\selectfont \begin{tabular}{c c c} 
\hline
Families for $f$ & Expression & Parameters \\
\hline
Dagum & $f(r)=1 - \left ( \frac{r^{\beta}}{1+r^{\beta}}\right )^{\tau}  $ & $\beta,\tau \in (0,1]$  \\
\hline
Mat{\'e}rn & Equation (\ref{matern}) & $0 < \nu \le 1/2 $ \\
\hline 
Gen. Cauchy & $f(r)=(1+r^{\alpha})^{-\beta/\alpha}$ & $\alpha \in (0,1]$, $\beta>0$ \\ 
\hline
Pow. Exponential & $f(r)=\exp(-r^{\alpha})$ & $\alpha \in (0,1]$ \\
\hline \end{tabular} }  \\[1mm]
\hline
Adaptive Gneiting & \multicolumn{1}{c}{Equation (\ref{PBG2})} & $f $ as in PBG family \\
\multicolumn{3}{|c|}{\linespread{2.2}\selectfont 
 \begin{tabular}{c c c} 
\hline
Families for $g_{[0,\pi]}$ & Expression & Parameters \\
\hline
Dagum & $g_{[0,\pi]}(r)=1 + \left ( \frac{r^{\beta}}{1+r^{\beta}}\right )^{\tau}  $ & $\beta,\tau \in (0,1]$  \\
\hline
Gen. Cauchy & $g_{[0,\pi]}(r)= (1+r^{\alpha})^{\beta/\alpha}$ & $\alpha \in (0,1]$, $\beta \le \alpha $ \\ 
\hline
Power & $g_{[0,\pi]}(r)=c+ r^{\alpha}$ & $\alpha \in (0,1]$, $c>0$ \\
\end{tabular} }   \\
\hline
PBG2 & Equation (\ref{wen-space-time}) &  $h$ positive, decreasing and  \\
&& convex with $\lim_{t \to \infty}h(t)=0$.\\
\hline 
Dynamical Wendland & \multicolumn{1}{c}{ Equation (\ref{wen-space-time3})} &  $h$ positive, decreasing and \\ 
& \multicolumn{1}{c}{} & convex with $\lim_{t \to \infty}h(t)=0$.\\
\multicolumn{3}{|c|}{\linespread{2.2}\selectfont \begin{tabular}{c c c} 
\hline
Families for $\varphi_{\mu,k}$ & Expression & Parameters \\
\hline
Dynamical Wendland 2 & $h(|u|)^{\alpha} \left ( 1- \frac{r}{h(|u|)}\right )_+^6 \left ( 1+ 6 \frac{r}{h(|u|)} \right )$ & $\alpha \ge 4 $ \\
\hline 
Dynamical Wendland 4 & $h(|u|)^{\alpha} \left ( 1- \frac{r}{h(|u|)}\right )_+^8  \left ( 1+ 8 \frac{r}{h(|u|)} + \frac{63}{3} \left (\frac{r}{h(|u|)} \right )^2 \right )$ & $\alpha \ge 4 $ \\
\end{tabular}}\\ 
\hline
\end{tabular}}
\end{table}

\subsection*{Appendix C. Mathematical Proofs} 

\subsubsection*{C1. The Quasi Arithmetic Class on Spheres} 
For ease of exposition, we slightly deviate from the notation of the paper and denote $f \circ g$ the composition of $f$ to $g$. Let us recall the expression of quasi arithmetic covariances as in Equation (\ref{qarf}):
$$ \psi(d_{{\rm GC}},u) = f \left ( \frac{1}{2} f^{-1} \circ \psi_{\bS}(d_{{\rm GC}}) +  \frac{1}{2} f^{-1} \circ C_{{\cal T}}(u) \right ), \qquad (d_{{\rm GC}},u) \in [0,\pi] \times \R. $$ 
By Bernstein's theorem \cite[p. 439]{Feller}, a function $f:[0,\infty) \to \R$ is completely monotonic if and only if 
\begin{equation} \label{bernstein}
f(t)= \int_{[0,\infty)} {\rm e}^{- \xi t} H({\rm d} \xi),  \qquad t \ge 0,
\end{equation}
where $H$ is a positive and bounded measure. A Bernstein function is a continuous mapping on the positive real line, having a first derivative being completely monotonic. We are now able to give a formal assertion for the validity of the quasi arithmetic construction. 

\begin{thm} \label{THM:QARF} 
Let $f:[0,\infty) \to \R_+$ be a completely monotonic function. Let $f_{1}:[0,\infty) \to \R$ be a continuous functions such that $f^{-1} \circ f_1$ is a Bernstein function. Let $C_{{\cal T}}: \R \to \R$ be a continuous, symmetric covariance function such $f^{-1} \circ C_{{\cal T}}$ is a temporal variogram. Call $\psi_{\bS}= f_{1,[0,\pi]}$ the restriction of $f_1$ to the interval $[0,\pi]$. Then,
\begin{equation}
\label{qarf2} \psi(d_{{\rm GC}},u) = {\cal Q}_{f} \left ( \psi_{\bS}(d_{{\rm GC}}), C_{{\cal T}}(u) \right ), \qquad (d_{{\rm GC}},u) \in [0,\pi] \times \R, 
\end{equation}
is a geodesically isotropic space-time covariance function on $\S^n \times \R$, for all $n=1,2,3,\ldots$. 
\end{thm}

\begin{proof}
Denote $g$ the composition $f^{-1} \circ f_1$ and call $g_{[0,\pi]}$ the restriction of $g$ to $[0,\pi]$, obtained through 
$g_{[0,\pi]}= f^{-1} \circ f_{1,[0,\pi]}$. By assumption, $g$ is a Bernstein function. Thus, arguments in \cite{porcu-schilling}, with the references therein, show that the function $h$, defined through
$$ h(t;\xi) = \exp(-\xi g(t) ), \qquad t, \xi \ge 0, $$
is a completely monotonic fuction for any positive $\xi$. Thus, $$h_{[0,\pi]}(d_{{\rm GC}};\xi) = \exp \left ( - \xi f^{-1} \circ f_{1,[0,\pi]} (d_{{\rm GC}})\right ), \qquad d_{{\rm GC}} \in [0,\pi],$$ is the restriction of a completely monotonic function to the interval $[0,\pi]$. Invoking Theorem~7 in \cite{gneiting2}, we obtain that $h_{[0,\pi]}$ is a geodesically isotropic covariance function on any $n$-dimensional sphere. Additionally, since $f^{-1} \circ C_{{\cal T}}$ is a temporal variogram, by Schoenberg's theorem \citep{schoenberg2} we deduce that $k(u;\xi)=  \exp(- \xi f^{-1} \circ C_{{\cal T}}(u)) $, $u \in \R$, is a covariance function on the real line for every positive $\xi$. Thus, the scale mixture covariance
\begin{eqnarray*} 
 \psi(d_{{\rm GC}},u) &=& \int_{[0,\infty)} h_{[0,\pi]}(d_{{\rm GC}};\xi) k(u;\xi) H({\rm d} \xi) \\
 &=& \int_{[0,\infty)}  \exp \left ( - \xi g_{[0,\pi]} (d_{{\rm GC}}) - \xi f^{-1} \circ C_{{\cal T}}(u) \right ) H({\rm d} \xi)  \\
 &=& {\cal Q}_f(\psi_{\bS} (d_{{\rm GC}}),C_{{\cal T}}(u)), \qquad (d_{{\rm GC}},u) \in [0,\pi] \times \R, 
 \end{eqnarray*}
is a covariance function on $\S^n \times \R$ for all $n \in \mathbb{N}$. 
\end{proof}

\subsubsection*{C2. Proofs for Section \ref{compact}}

\begin{thm} \label{THM:ASKEY}
Let $h$ be a positive, decreasing and convex function on the positive real line, with $h(0)=c$, $0<c\le \pi$, and $\lim_{t \to \infty} h(t)=0$. Let $\alpha \ge 3 $ and $\mu \ge 4 $. Then, Equation (\ref{wen-space-time}) defines a geodesically isotropic and temporally symmetric covariance function on $\S^3 \times \R$. 
\end{thm}

\begin{proof}
The proof is based on scale mixture arguments in concert with the calculations in \cite{Porcu-Bevilacqua-Genton-2}. In particular, we have that the function in Equation (\ref{wen-space-time}) is the result of the scale mixture of the function $\psi_{\bS}(d_{{\rm GC}};\xi)= (1-d_{{\rm GC}}/\xi)_+^{n}$ with the function $C_{{\cal T} }(u;\xi) = \xi^{n} \left ( 1- \frac{\xi}{h(|u|)} \right )_+^{\gamma}$, $\xi > 0, t \ge 0$,  $\gamma \ge 1, n \ge 2$. In particular, arguments in Lemmas 3 and 4 in \cite{gneiting2} show that $\psi_{\bS}$ is a covariance function on $\S^3$ for every positive $\xi$. Under the required conditions on the function $h$, we have that $C_{{\cal T}}$ is a covariance function on $\R$ for every positive $\xi$. Thus, the scale mixture arguments in \cite{Porcu-Bevilacqua-Genton-2}, with $\mu=n+\gamma+1$ and $\alpha=n+1$ can now be applied, obtaining the result.  
\end{proof}

A more sophisticated argument is required to show that the structure in Equation (\ref{wen-space-time3}) is positive definite on the circle cross time. We start with the proof of Theorem \ref{g_wendland2} because its arguments will be partially used to prove Theorem \ref{gen_w}. 

\begin{proof}[Proof of Theorem \ref{g_wendland2}]
Denote $\varphi_{[0,\pi]}$ the restriction of $\varphi$ to the interval $[0,\pi]$ with respect to the first argument. Let $n \in \mathbb{N}$. Consider the sequence of functions $b_{n}(\cdot)$, defined through 
\begin{eqnarray} \label{fourier}
b_{n}(u) &=&  \frac{2}{\pi} \int_{0}^{\pi} \cos(n z) \psi(z,u) {\rm d } z= \frac{2}{\pi} \int_{0}^{\pi} \cos(n z) \varphi_{{[0,\pi]}}(z,u) {\rm d } z  \nonumber \\
&=& \frac{1}{\pi} \int_{-\infty}^{\infty} \cos(n x) \varphi(x,u) {\rm d }x .
\end{eqnarray}
Since $\varphi$ is positive definite on $\R \times \R$, arguments in Lemma 1 in \cite{gneiting1} show that $b_{n}(u)$ is a covariance function on $\R$ for all $n \in \mathbb{N}$. Additionally, we have 
$$ \sum_{n=0}^{\infty} b_{n}(0) = \sum_{n=0}^{\infty} \frac{1}{\pi} \int_{-\infty}^{\infty} \cos(n x) \varphi(x,0) {\rm d} x=  \sum_{n=0}^{\infty} \widehat{\varphi}_{\mathbb{N}}(n) < \infty , $$ where $\widehat{\varphi}_{\mathbb{N}}$ denotes the Fourier transform of $\varphi(x,0)$ restricted to natural numbers. Thus, we get that $\sum_n b_{n}(0)< \infty$ because the Fourier transform of a positive definite function is nonnegative and integrable. We can thus invoke Theorem 3.3 in \cite{berg-porcu} to obtain that $\psi(d_{{\rm GC}},u)=\varphi_{[0,\pi]}(d_{{\rm GC}},u)$ is a covariance function on the circle $\S^1$ cross time. 
\end{proof}

\begin{proof}[Proof of Theorem \ref{gen_w}]
We give a proof of the constructive type. Let $k \in \mathbb{N}$. Arguments in Theorem 1 of \cite{Porcu-Bevilacqua-Genton-2} show that the function
\begin{equation} \label{wen-space-time4}
C(x,u)= \sigma^2 h(|u|)^{\alpha} \varphi_{\mu,k} \left ( \frac{|x|}{ h(|u|)} \right ), \qquad (x,u) \in \R \times \R, 
\end{equation}
is positive definite on $\R \times \R$ provided $\alpha \ge 2k +3$ and $\mu \ge k+4$. Arguments in \cite{Porcu-Zastavnyi-Bevilacqua} show that 
\begin{eqnarray*}
 b_{n}(u) &=& \int_{_\infty}^{-\infty} \cos (n x) C(x,u) {\rm d} x \\ &\propto & \sigma^2  h(|u|)^{\alpha+d} \mathstrut_1F_{2}\bigg ( 1+k; \frac{\mu+ 2}{2}+k, \frac{\nu+ 3}{2}+k; -n^2 h(|u|)^2/4 \bigg ), \qquad u \in \R, \quad n \in \mathbb{N}. 
\end{eqnarray*}
From the argument in (\ref{fourier}) in concert with Lemma 1 in \cite{gneiting1}, we have that $b_n(u)$ is positive definite on the positive real line for each $n$. Additionally, arguments in \cite{Porcu-Zastavnyi-Bevilacqua} show that, for $\mu\ge k+4$, $b_n(u)$ is strictly decreasing in $n$. Application of Proposition 3.6 of \cite{berg-porcu} shows that (\ref{wen-space-time3}) is positive definite on $\S^3 \times \R$. 
\end{proof}

\section*{Acknowledgments}
We gratefully acknowledge the NASA scientists responsible for MERRA-2 products. 

\small 

\bibliographystyle{mywiley}
\bibliography{mybib}{}

\end{document}